\documentclass[11pt]{amsart}
\usepackage{fullpage,verbatim,amssymb}
\usepackage{hyperref}
\usepackage[usenames,dvipsnames]{color}
\usepackage{enumitem, soul}
\usepackage{amsthm}
\theoremstyle{plain}

\newtheorem*{theorem*}{Theorem}

\makeatletter
\let\@@pmod\mod
\DeclareRobustCommand{\mod}{\@ifstar\@pmods\@@pmod}
\def\@pmods#1{\mkern4mu({\operator@font mod}\mkern 6mu#1)}
\makeatother

\definecolor{blue}{rgb}{0,0,1}
\definecolor{red}{rgb}{1,0,0}
\definecolor{green}{rgb}{0,.6,.2}
\definecolor{purple}{rgb}{1,0,1}

\long\def\red#1\endred{\textcolor{red}{#1}}
\long\def\blue#1\endblue{\textcolor{blue}{#1}}
\long\def\purple#1\endpurple{\textcolor{purple}{ #1}}
\long\def\green#1\endgreen{\textcolor{green}{#1}}

\newcommand{\sm}{\left(\begin{smallmatrix}}
\newcommand{\esm}{\end{smallmatrix}\right)}
\newcommand{\bpm}{\begin{pmatrix}}
\newcommand{\ebpm}{\end{pmatrix}}

\newcommand{\Z}{\mathbb{Z}}

\let\Re\relax
\DeclareMathOperator{\Re}{Re}

\newtheorem{thm}{Theorem}

\newtheorem{theorem}{Theorem}

\newtheorem{corollary}[theorem]{Corollary}

\theoremstyle{remark}

\newtheorem{remark}[theorem]{Remark}

\numberwithin{theorem}{section}
\numberwithin{equation}{section}

\title{Asymptotics of the spectral determinant of the weighted Laplacian for a sequence of compact Riemann surfaces of infinitely growing volume}

\author{Jay Jorgenson}
\thanks{The first named author acknowledges grant support from PSC-CUNY Awards 67415-00 55 and 68462-00 56, which are jointly funded
by the Professional Staff Congress and The City University of New York.}
\author{Lejla Smajlovi\'{c}}
\author{Polyxeni Spilioti}
\thanks{The third named author was supported by the Hellenic Foundation for Research and Innovation (H.F.R.I.) under the ``3rd Call for H.F.R.I. Research Projects to support Faculty Members $\&$ Researchers"(Project Number: 25622).}

\begin{document}
\maketitle

\begin{abstract}\noindent
Let $(X,\chi,k)$ be a triple consisting of 
a smooth, compact hyperbolic Riemann surface $X$ of genus $g$, and an $m$ dimensional unitary multiplier system $\chi$ of admissible weight $k$.
Our first result establishes an analogue of the prime geodesic theorem for the weighted prime geodesic counting function associated to
$(X,\chi,k)$.  The error term we obtain is explicit with effectively computable constants which depend solely
on the genus of $X$, the dimension of $\chi$, the length of shortest geodesic on $X$ and the smallest non-zero eigenvalues of the weighted Laplacian $\Delta_{2k}$ as well that of 
the scalar Laplacian $\Delta_{0}$.
Our second result studies the asymptotic behavior of the spectral determinant $\det\Delta_{2k_n}$ for a sequence $(X_{n}, \chi_{n}, k_{n})$
for which the genus of $X_n$ tends to infinity. Under reasonably general circumstances, namely the existence of a weak spectral gap, a uniform discreteness of the
underlying Fuchsian group, and a type of non-accumulation of bounded geodesics, we prove that $\log\det\Delta_{2k_n}/\mathrm{vol}(X_{n})$ converges to a constant $C_{\alpha}$ which depends only on $\alpha=\lim_{n\to\infty} k_n$.  Our result is deterministic
and is compatible with the three well-studied probabilistic models, namely Weil-Petersson, 
Brooks-Makover, and random covers model.

\end{abstract}

\section{Introduction}
Let $X$ be a smooth, compact hyperbolic Riemann surface of genus $g$. There exists a discrete subgroup $\Gamma<{\rm SL}_2(\mathbb R)$ 
with projection $\overline\Gamma$ onto ${\rm PSL}_2(\mathbb R)$
with $-I_2\in\Gamma$, where $I_2$ is the identity matrix, and where $\Gamma$ acts on the hyperbolic upper half
plane $\mathbb{H}$ such that $X$ can be identified with the quotient space $\overline\Gamma \backslash \mathbb{H}$. 
Consider an $m$ dimensional 
unitary multiplier system $\chi$ on $\Gamma$, which equals a product of an $m$-dimensional unitary representation 
$\tilde\chi$ of the fundamental group $\pi_{1}(X)$ of $X$, and a multiplier system $v$ on $\Gamma$ of admissible weight $2k$.    Associated to
$(X,\chi,k)$, we study two separate but closely related mathematical objects:  The prime geodesic counting function, and the
spectral determinant.   We now will describe our results for each of these objects.

\subsection{The prime geodesic counting function}
Any non-identity element
$P \in \Gamma$ can be written as $P = P_{0}^{n}$ for some positive integer $n$ and $P_{0} \in \Gamma$ is
primitive, meaning $P_{0}$ is not a positive power any other element in $\Gamma$.  The (classical) prime geodesic theorem is the study of
the asymptotic behavior of the counting function
$$
\pi(x)= \sum_{P_{0}\in\Gamma,\, {\rm Tr}(P_{0}) >2: \, N(P_{0})\leq x} 1
$$
as $x$ tends to infinity,
where the sum is over all primitive $P_{0} \in \Gamma$ and where $N(P_{0})$ denotes the norm of $P_{0}$.
For general $X$, the best known result is that
\begin{equation}\label{eq:classical_PGT}
\pi(x) = \text{\rm li}(x) + \sum_{3/4 < s_{j}<1}\text{\rm li}(x^{s_{j}}) + O\left(x^{3/4}/\log x\right)
\end{equation}
where $s_{k}(1-s_{k}) = \lambda_{k}$ is a eigenvalue of the hyperbolic Laplacian $\Delta_{0}$ which acts on smooth functions on $X$
and where $\text{\rm li}(x)$ is the logarithmic integral function; see page 257 of \cite{Buser92}.

It is immediate that the prime geodesic theorem is a geometric analogue of the prime number theorem, which is a
mathematical question of unparalleled significance since the conjectured error term for the prime number theorem
is one manifestation of the Riemann hypothesis.  The study of the prime geodesic theorem began with Huber and
Selberg for general $X$, and there have been many recent developments when $X = \Gamma \backslash \mathbb{H}$ is the quotient
of the upper half plane $\mathbb{H}$ by an arithmetic group $\Gamma$.  While it has been shown that the prime
geodesic theorem is related to other areas in mathematics such as dynamics and quantum chaos, one can follow the
(implicit) point of view of Huber and Selberg that, perhaps, further understanding of the prime geodesic theorem
may provide insight into the Riemann hypothesis.

Using well-studied number theoretic
techniques, the study of $\pi(x)$ can be shown to be equivalent to the study of the asymptotic behavior in $x$ of
$$
\Psi(x)= \sum_{P\in\Gamma,\, {\rm Tr}(P) >2: \, N(P)\leq x} \Lambda(P)
\,\,\,\,\,
\text{\rm where}
\,\,\,\,\,
\Lambda(P) := \frac{\log N(P_{0})}{1-N(P)^{-1}};
$$
see, for example, page 82 of \cite{He}.  With this in mind,
our first result in this article is the following analogue of  \eqref{eq:classical_PGT}
associated to the triple $(X,\chi,k)$.

\begin{thm} Associated to $(X,\chi,k)$, let
$$
\Psi(x,\chi)= \sum_{P\in\Gamma,\, {\rm Tr}(P) >2: \, N(P)\leq x} \Lambda(P)\mathrm{Tr}(\chi(P))
$$
Let $\Delta_{2k}$ be the weighted hyperbolic Maass-Laplacian, defined by \eqref{eq. defn weighted Lapl} below, which has corresponding eigenvalues
$\lambda_{j}$ of multiplicity $m(\lambda_{j})$, and write $\lambda_j = s_{j}(1-s_{j})$.
Then,
\begin{equation}\label{eq. distr of Psi of chi}
  \Psi(x,\chi)= \sum_{\lambda_j\leq 1/4}m(\lambda_j)\frac{x^{s_j}}{s_j} + \mathcal{E}(x; X,\chi)
\end{equation}
with error $\mathcal{E}(x; X,\chi)$ which can be estimated in either of the following two ways.
\begin{enumerate}[label=\alph*)]
\item  Let ${\rm sys}(X)$
be the systole on $X$ (meaning the length of the shortest closed geodesic), and let $\mathcal N_k$ (resp. $\mathcal N_0$) denote number of eigenvalues
which are $\leq 1/4$ of the Laplacian $\Delta_{2k}$ (resp. $\Delta_0$).  Then for $x\geq 2$,  $\mathcal{E}(x; X,\chi,x) = g(X,\chi)x^{3/4}$
where
\begin{equation}\label{eq:PGT_firstbound}
|g(X,\chi)|\leq C_1\frac{m(g-1)}{{\rm sys}(X)} + (\mathcal N_k -1) \left(\frac{C_2}{1-s_1}+C_3 \right) +
(\mathcal N_0 -1) \left(\frac{C_4}{1-\widetilde s_1}+C_5 \right),
\end{equation}
for some absolute constants $C_1,\, C_2,\, C_3, \, C_4, \, C_5$ independent of $X$ and $\chi$. 
\item
There is an explicitly computable, universal, constant $x_{0}$ such that for $x > x_{0}$ we have
\begin{equation}\label{eq:PGT_secondbound}
|\mathcal E(x; X,\chi)|\leq \mathcal C_1mg x^{5/6} + \mathcal C_2 mg x^{2/3} \max\left\{ 0, \log\left(\frac{1}{{\rm sys}(X) x^{1/6}}\right)\right\}.
\end{equation}
\end{enumerate}
\end{thm}

\vskip .10in
Part b) of the above result generalizes Theorem 2 of \cite{WX25} who considered the case when $\chi$ is the 
trivial representation and with multiplier weight $k=0$. Part a) is related to main theorems of \cite{FJK11} and \cite{Av21} 
where an explicit error term for the prime geodesic theorem in the setting of co-finite, not necessarily co-compact, 
$\overline\Gamma$, identity representation, and $k=0$. The implied constant in the error term from \cite{FJK11} and \cite{Av21} 
depends upon the diameter of the surface which can grow much faster than its volume. For that reason these results could not be 
used to study behavior of the spectral determinant as the volume grows to infinity. In Theorem A, we did not try to optimize 
implied constants. Rather, our aim was to deduce an upper bound for the constant in the error term which is independent upon 
the diameter. It is possible to deduce better bounds for the implied constants in Theorem A a) using the same method. We 
decided not to pursue that because such a result is not needed for our main application, which is to study sequences of spectral determinants.

\subsection{Spectral determinants}

As stated, let $\Delta_{2k}$ be the weighted Maass-Laplacian of admissible weight $2k$, meaning that $km(2g-2)\in\mathbb Z$. The study of the weighted Maass-Laplacian is relevant from the perspective of automorphic forms (see \cite{Fa77} or \cite[pp. 481--496]{Hejhal83}) where by varying the weight $2k$ allows one to derive valuable conclusions about spaces of automorphic forms. An interested reader is referred to Section 1 of \cite{KetAll21} and references therein for a historical background and overview of  applications of $\Delta_{2k}$ to automorphic forms.

Moreover, the weighted Maass-Laplacian in dimension $m=1$ is closely related to the magnetic Laplacian  which is an operator that models the quantum mechanics of a charged particle moving on a surface under the influence of an external, constant magnetic field $F=Bd\mu$, where $d\mu$ is the volume element on $X$, see \cite{CL26, Co86,  KT19, KT22}. Namely, the magnetic Laplacian $\Delta^B$ is well defined if and only if the Dirac quantization condition, stating that the integral of the magnetic field $Bd\mu$ over the entire surface must be an integral multiple of $2\pi$ is fulfilled, thus forcing $(2g-2)B\in\mathbb Z$\footnote{See also \cite{Ta24} for both mathematical and physical analysis of this fact.}. If $B$ satisfies the Dirac quantization condition, then $\Delta^B = \Delta_{2B} + B^2$, meaning that the spectral properties of $\Delta_{2B}$ are closely related to those of $\Delta^B$.

 Let $\{\lambda_{k}\}$ be the sequence
of eigenvalues of $\Delta_{2k}$, and let $\zeta_{X,\chi,k}(s)= \sum \lambda_{k}^{-s}$ be the spectral zeta function which is
defined from the subset of non-zero eigenvalues.  After one proves that $\zeta_{X,\chi,k}(s)$ admits a meromorphic continuation
from $\text{\rm Re}(s) > 1$ to all $s \in \mathbb{C}$ and is holomorphic at $s=0$.  Then, 
one defines the determinant of the Laplacian $\mathrm{det}(\Delta_{2k_n})$ of
$\Delta_{2k}$, or the spectral determinant, by
$$
 \mathrm{det}(\Delta_{2k}) = -\zeta'_{X,\chi,k}(0).
$$
Rather than study the spectral determinant for a fixed triple $(X,\chi,k)$, we consider a sequence
$(X_{n},\chi_{n},k_{n})$ under certain assumptions which, in somewhat general terms, are as follows.

\medskip\noindent
({\bf Weak spectral gap}) Let $\mathcal N_{k,n}$ be the number of eigenvalues of
  $\Delta_{2k_n}$ which are less than or equal to $1/4$, and let $\lambda_{1,n}$ be  the smallest non-zero eigenvalues.  Then, we assume there is a
$\beta < \infty$ independent of $n$ such that
\begin{equation}
\limsup_{n\to\infty} \frac{\mathcal N_{k,n}}{\lambda_{1,n}{\rm vol}(X_n)}=\beta.
\end{equation}

\medskip\noindent
({\bf Uniform discreteness}) Let
\begin{equation}
    1<\delta_n = \min\limits_{P\in\Gamma_n,\, P\neq \mathrm{Id}}\{N(P)\}.
\end{equation}
    Then there exists $\eta>1$ independent of $n$ such that $\delta_n\geq \eta$ for all $n\geq 1$.

\medskip\noindent
({\bf Non-accumulation of small geodesics})
For some constants $C$, $L$ and $0 < \kappa < 1$, we have for all $n$ the bound
$$
N_{X_n}(L):=|\{P\in\Gamma_n: N(P)\leq L\}| \leq C{\rm vol}(X_n)^{\kappa},
$$
where $|A|$ denotes the cardinality of the finite set $A$.  Following \cite{naud2023determinants}, we refer to this
assumption as $\mathcal{H}(C,L,\kappa)$.

In Remarks \ref{rm:det_and_random} and \ref{rm:unif disc} below, we discuss rationale behind the first 
two assumptions and their compatibility with models of random compact Riemann surfaces. In \cite{naud2023determinants}, 
Naud proved that the third assumption holds true with probability one in Weil-Petersson, Brooks–Makover, and random covers model.

\medskip
With these assumptions, we have the following result.

\begin{thm} Let $\{X_n\}$ be a sequence of compact Riemann surfaces of finite volume satisfying the
weak spectral gap and uniform discreteness assumptions.
Let $\{\chi_n\}$ be an associated sequence of $m$ dimensional unitary multiplier systems of weights $2k_n\in(0,2)$, and assume that
$\lim\limits_{n\to \infty} k_n =\alpha\in[0,1)$.
If $\alpha =0$, we further assume that the multiplicities $m(\lambda_{0,n})$ of the first eigenvalue of the Laplacian $\Delta_{2k_n}$
on $X_{n}$ are bounded uniformly in $n$.  Assume there is universal constant $C$ such that
for all $\varepsilon>0$ there exists an $L_\varepsilon >0$, of order $\log L_{\varepsilon} = O(1/\varepsilon)$ (and will be determined
precisely below) for which $\mathcal{H}(C,L_\varepsilon,\kappa)$ holds.  Then, with all this, there is an explicitly computed constant $C_{\alpha}$, defined in \eqref{eq:alpha_constant} below, such that
for $n$ sufficiently large, we have that
\begin{equation}\label{eq:det_asymptotics}
\bigg| \frac{2\pi \log\mathrm{det}(\Delta_{2k_n})}{m{\rm Vol}(X_{n})}  - C_{\alpha}\bigg| <\varepsilon.
\end{equation}
\end{thm}

\medskip
In the case when $\lim_{n\to\infty}k_{n} = 0$ and $m=1$, our result specializes to the main theorem in \cite{naud2023determinants}.  Beyond this example, we are able to specialize to new instances not considered elsewhere, such
as when $k_{n} = p/q$ for some rational $p/q \in (0,2)$ and $g_{n}$ is $1$ mod $q$ and tends monotonically to infinity.
We refer to the corollaries below for this results and additional corollaries.

The asymptotic result \eqref{eq:det_asymptotics} is a deterministic result that is more general than
the one obtained by
simply considering a sequence $(\Gamma_{n})_{n\in\mathbb{N}}$,
where $\Gamma_{n+1}$ is a finite index subgroup of $\Gamma_{n}$, though such a setting
is one possibility.  Also, our result is deterministic, but it is, indeed, consistent with the three main models of random surfaces;
see Remarks \ref{rm:det_and_random} and \ref{rm:unif disc} for further discussion.

\subsection{Organization of the article}

In section \ref{sec: prelim}, we establish notation and recall results from the literature.  It was our aim to make the
article as self-contained as possible, so we have included a through discussion of preliminary material.

In section \ref{sec:pseudo_primes}, we prove \eqref{eq:PGT_firstbound} using the trace formula and following the approach of
\cite{FJK11} which focuses on developing explicit and effective bounds; a refinement of the methods from \cite{FJK11}
is given in \cite{Av21}.  Initially, we obtain an effective, albeit somewhat
imprecise, bound for Weyl's law; see \eqref{eq. number eigenv trivial bound 1}.  From the Weyl's law bound, we first prove
\eqref{eq:PGT_firstbound} when $\chi$ is the one dimensional trivial representation and $k=0$; see \eqref{eq. pre-final for id}
and \eqref{eq. sum with -1}.  The preliminary result when $\chi$ is trivial gives important bounds for the contribution to the
trace formula from the hyperbolic elements, from which we prove the general result \eqref{eq:PGT_firstbound}.

The bound \eqref{eq:PGT_secondbound} is proved in section \ref{sec:uniform_PGT}.  Again, the result follows from the trace formula
where in this setting we are employing aspects of the approach to Theorem 2 from \cite{WX25}.  As in section \ref{sec:pseudo_primes},
we first prove \eqref{eq:PGT_secondbound} when $\chi$ is trivial.  In order to extend the analysis to general $\chi$ and $k$, we use
a critical lemma from \cite{He}; see \eqref{eq:hejhal_bound}.  Specifically, the series which defines $\Psi(x,\chi)$ is, in fact,
equal to the series obtained by replacing the factor $\mathrm{Tr}(\chi(P))$ by $\mathrm{Re}\left(\mathrm{Tr}(\chi(P))\right)$.  In doing so,
we get a new, real-valued series \eqref{eq:real_series} which we can analyze using the information derived in case $\chi$ is trivial, and using trace formulas for trivial and non-trivial $\chi$. By combining those results we deduce the bound  \eqref{eq:PGT_secondbound}.

In section \ref{sec:spectral_dets}, we use the asymptotic formula \eqref{eq. distr of Psi of chi} and the bound  \eqref{eq:PGT_secondbound} to prove \eqref{eq:det_asymptotics}.
To begin, we use a relation between the Selberg zeta function and the regularized determinant to reduce the question to the asymptotic analysis of the
special value of the Selberg zeta function at $s=1$;  see \eqref{e:det_Delta}.  Using \eqref{eq. distr of Psi of chi} and \eqref{eq:PGT_secondbound}, the
special value of the Selberg zeta function can be expressed as a Stieltjes integral involving $\Psi(x,\chi)$; see \eqref{eq. logZ at 1}.
The aforementioned asymptotic conditions (weak spectral gap, uniform discreteness, and non-accumulation of small geodesics) then
are employed to complete the proof.

\section{Preliminaries}\label{sec: prelim}

\subsection{Basic notation}\label{sec:notation}

Let $\Gamma< {\rm SL}_2(\mathbb{R})$ be a discrete subgroup acting totally discontinuously on the upper half plane
$\mathbb{H}=\{x+iy:\, x, y\in\mathbb{R},\, y>0\}$, such that $-I_2\in \Gamma$, where $I_2$ is the identity matrix.
Let $\mathcal{F}$ be a (Ford) fundamental domain for the action of $\Gamma$ on $\mathbb{H}$.

Let $\overline{\Gamma}$ be the projection of $\Gamma$ into $\mathrm{PSL}(2,\mathbb{R})$ and
$X$ be the Riemann surface associated to the quotient
space $\overline{\Gamma}\backslash\mathbb{H}$. We assume that $X$ is smooth and compact, which is equivalent to saying
that $\Gamma$ is Fuchsian group of the first kind containing only hyperbolic elements. Let $g>1$ be the genus of
$\overline{\Gamma}\backslash\mathbb{H}$ \footnote{We will also say that $g>1$ is a genus of $\Gamma$.}. The hyperbolic
area of  $\mathcal{F}$, or equivalently, the volume of $X$ is
\begin{equation}\label{eq: volume of F}
\omega(\mathcal{F})={\rm vol}(X) = 2\pi(2g-2).
\end{equation}

\subsection{Unitary multiplier systems}\label{sec: mult systems}

For a positive integer $m$ we fix an $m\times m$ unitary multiplier system $\chi$ on $\Gamma$ of arbitrary admissible 
weight $2k\notin 2\Z$. By this we mean a product $\chi=v \cdot \tilde\chi$ of an irreducible $m\times m$ unitary 
representation $\tilde \chi: \overline\Gamma \to V$, where $V$ is an $m$-dimensional Hermitian vector space, and a 
weight $2k$ multiplier system. Note that the product $v \cdot \tilde\chi$ is well defined on $\Gamma$ because for 
$\gamma\in\Gamma$ either $\gamma$ or $-\gamma$ belong to $\overline \Gamma$. Therefore, for any $\gamma\in\Gamma$, we have $\chi(\gamma)=\tilde\chi(\tilde\gamma)v(\gamma)$, where $\tilde\gamma$ is the projection of $\gamma\in\Gamma$ to 
${\rm PSL}_2(\mathbb R)$.

Let us recall from pp. 331--337 of \cite{Hejhal83}
that the weight $2k$ multiplier system is a map $v:\Gamma \to S^1$, where $S^1=\{z\in\mathbb C:\,\ |z|=1\}$ such that
\begin{equation}\label{eq: multsysprop1}
v(-I_2)=e^{-2\pi i k},
\end{equation}
and, for any $\gamma$ and $\eta$ in $\Gamma$,  we have that
\begin{equation}\label{eq: multsysprop2}
v(\gamma\eta) =\sigma_{2k}(\gamma,\eta)v(\gamma)v(\eta),
\end{equation}
where $\sigma_{2k}$ is a weight $2k$ factor system, whose definition from
\cite[p. 332]{Hejhal83} we now recall.
Write
$$
\eta = \begin{pmatrix}\eta_{1} & \eta_{2} \\ \eta_{3} & \eta_{4}
\end{pmatrix}
\,\,\,\,\,
\text{\rm and}
\,\,\,\,\,
\gamma = \begin{pmatrix}\gamma_{1} & \gamma_{2} \\ \gamma_{3} & \gamma_{4}
\end{pmatrix}
\,\,\,\,\,
\text{\rm as well as }
\,\,\,\,\,
\gamma\eta = \begin{pmatrix}\delta_{1} & \delta_{2} \\ \delta_{3} & \delta_{4}
\end{pmatrix}.
$$
Then, for any $z \in \mathbb{H}$, we have that
$$
\eta z = \frac{\eta_{1}z + \eta_{2}}{\eta_{3}z + \eta_{4}}
\,\,\,\,\,
\text{\rm and}
\,\,\,\,\,
\gamma_{3}\eta z + \gamma_{4} = \frac{\delta_{3}z + \delta_{4}}{\eta_{3}z+\eta_{4}}.
$$
Then,
$\sigma_{2k}(\gamma,\eta) = \exp(2\pi ik w(\gamma,\eta))$ where
$$
w(\gamma,\eta) = \text{\rm arg}(\gamma_{3}\eta z + \gamma_{4})
+ \text{\rm arg}(\eta_{3}z + \eta_{4}) -  \text{\rm arg}(\delta_{3}z + \delta_{4})
$$
is independent of $z$ and is
an integer from the set $\{-1,0,1\}$.  Various co-cycle relations
satisfied by $w(\gamma,\eta)$ are given on page 18 of \cite{FiBook87}.

The weight $2k$ is \emph{admissible} if there exists a multiplier system with weight $2k$.
The admissible values of $k$ are the real numbers
which lie in the set
\[A_m(\Gamma):=\frac{1}{m} \frac{2\pi}{\omega(\mathcal{F})}\mathbb{Z}= \frac{\mathbb{Z}}{2m(g-1)};\]
see \cite[Proposition 2.3, p. 335]{Hejhal83}.
Trivially, for any $a\in \Z$, the real number $a/(m(g-1))$ is an admissible weight for a unitary multiplier system of dimension $m$.

\subsection{Weighted Laplacian}\label{ss:Laplace}
Let $\Delta_{2k}$ be the weight $2k$ Maass-Laplacian which is defined as
\begin{equation}\label{eq. defn weighted Lapl}
  \Delta_{2k} = -y^2\left(\frac{\partial^2}{\partial x^2} + \frac{\partial^2}{\partial y^2} \right) +2kiy \frac{\partial}{\partial x}.
\end{equation}
The operator $\Delta_{2k}$ acts on the space of twice continuously differentiable functions $f:\mathbb{H}\to V$ which
satisfy the transformation property that
\begin{equation}\label{eq: transf prop functions}
f(\gamma z) = \exp(2ik \arg(cz +d)) \chi(\gamma) f(z) \quad \text{ for all }
 \gamma = \bpm * & * \\ c & d\ebpm \in \Gamma.
\end{equation}
Here we choose $\arg(cz+d) \in (-\pi, \pi]$.
We identify $\Delta_{2k}$ with its self-adjoint extension to the Hermitian space of all $L^2$ functions on $\mathcal{F}$ satisfying the
transformation property \eqref{eq: transf prop functions}.
The operator $\Delta_{2k}$ has only the discrete spectrum with eigenvalues
$$
\lambda_0=|k|(1-|k|) \leq \lambda_1\leq \cdots \leq \lambda_j  \leq  \cdots
$$
which tend to $+\infty$ (\cite[p. 370]{Hejhal83}).
For each $\lambda_{j}$ we define $s_{j}$ as a solution to the equation $\lambda_{j} = s_j(1-s_j)$.

\subsection{The gamma function and Barnes double gamma function}
The Barnes
double Gamma function, also called the $G$-function, $G(s)$ is defined by the product expansion
\begin{equation*}\label{e:barnesG}
G(s+1) = (2\pi)^{\frac{s}{2}} e^{-\frac{s(s+1)}{2}-\frac{1}{2}\gamma s^2} \prod_{n=1}^\infty \bigg\{\bigg(1+\frac{s}{n}\bigg)^n e^{-s+\frac{s^2}{2n}}\bigg\},
\end{equation*}
where $\gamma$ is the is the Euler-Mascheroni constant.  It can be shown that
\begin{equation}\label{eq. funct we. Barnes}
G(s+1) = \Gamma(s)G(s),
\end{equation}
where $\Gamma(s)$ is the classical Gamma function; see \cite[\S5.17]{dlmf}.
Both the Gamma function $\Gamma(s)$ and the Barnes $G$-function $G(s)$ appear as factors in the functional equations for the
Selberg zeta function.  The function $\psi(s) = \Gamma'(s)/\Gamma (s)$ denotes the logarithmic derivative of $\Gamma(s)$.

As is common throughout the literature, we will use $\Gamma$ to denote either the Gamma function or a discrete group, as described above.  We will make it clear in context precisely what is being signified by $\Gamma$.

\subsection{The Selberg zeta function}\label{sec:Selberg_and_Ruelle}

Following \cite[p.496]{Hejhal83}, the Selberg zeta function associated to the multiplier system
$\chi$ is defined for $\Re(s)>1$ by the absolutely convergent product
\begin{equation}\label{e:Z_def}
Z(s; \chi) = \prod_{\substack{P_0, \\ {\rm Tr}(P_0)>2}} \prod_{\ell=0}^\infty \det\left(I_m - \chi(P_0) N(P_0)^{-s-\ell}\right);
\end{equation}
the product runs through all primitive hyperbolic elements $P_0$ of the group $\Gamma$, and $N(P_0)$ is the norm of the element $P_0$.
As above, $m$ is the dimension of the multiplier system $\chi$.
The Selberg zeta function is studied in great detail in \cite{Hejhal83}.
Specifically, it is shown that the Selberg zeta function admits a meromorphic continuation to all $s \in \mathbb{C}$.
Furthermore, the Selberg zeta function admits a functional equation which relates the value at $s$ to the value at $1-s$ (\cite[\S 5]{Hejhal83}).

\subsection{The trace formula}\label{ss:traceformula}
The analysis we undertake involves an application of the Selberg trace formula.
In the generality required for our work, the form of the
trace formula required was developed in \cite{FiBook87} and \cite{Hejhal83}.  For the convenience of the reader,
we repeat here the formulation as stated in \cite{Gong95}.

Let $h(z)$ be an even function which is holomorphic in the strip $\vert \text{\rm Im}(z) \vert < \text{\rm max}(1/2,\vert k \vert -1/2 + \delta)$
for some $\delta > 0$.  Also, assume that $\vert h(z) \vert = O(\text{\rm Re}(z) ^{-2-\delta})$ as $\vert \text{\rm Re}(z)\vert \rightarrow \infty$
in the horizontal strip.  Let $g$ be the Fourier transform of $h$, normalized so that
$$
g(u) = \frac{1}{2 \pi}\int\limits_{-\infty}^{\infty}h(z) e^{-izu}du.
$$
Recall that the set of eigenvalues $\{\lambda_j\} $ for integers $j \geq 0$ of the weighted Laplacian
are such that $\lambda_j=s_j(1-s_j)=\frac{1}{4}+r_j^2$. Then, with the notation as above,
the Selberg trace formula in the setting we are considering is the identity,
as stated on page 440 of \cite{Gong95}, that
\begin{align}\label{eq:stf_id1}
\sum\limits_{n\geq 0}h(r_{n}) &= \frac{m\omega(\mathcal{F})}{4\pi}\int\limits_{-\infty}^{\infty}rh(r)\frac{\sinh(2\pi r)}{\cosh(2\pi r) + \cos (2\pi k)}dr
\\& \label{eq:stf_id2} +\frac{m\omega(\mathcal{F})}{2\pi}\sum\limits_{0 \leq l < \vert k \vert -1/2}(\vert k \vert - l - 1/2)h(i(\vert k \vert -l - 1/2))
\\&\label{eq:stf_hyp1} +\sum\limits_{P, \text{\rm Tr}(P)>2}\text{\rm Tr}(\chi(P))\frac{\log N(P_{0})}{N(P)^{1/2}-N(P)^{-1/2}}g(\log N(P)).
\end{align}

The left-hand side of \eqref{eq:stf_id1} is the spectral side of the trace formula. Its right-hand-side, and \eqref{eq:stf_id2} stem from the identity element
in the discrete group $\Gamma$, while line \eqref{eq:stf_hyp1} contains information associated
to the hyperbolic elements of $\Gamma$.  The sum in \eqref{eq:stf_hyp1} is over all hyperbolic elements $P$, and where $P = P_{0}^{n}$ for some
primitive hyperbolic element $P_{0}$ and positive integer $n$.

\subsection{Zeta regularized determinants}\label{sec:regularizedZeta}
Following \cite[p.441-442]{Gong95}, we define the spectral zeta function $\zeta(w, s)$ for $\Re(s)>1$ and $\Re(w)$ sufficiently large by
\begin{equation}\label{e:spectralzeta_def}
\zeta(w, s) = \sum_{j\geq 0} (s_j(1-s_j) - s(1-s))^{-w} .
\end{equation}
As proved in \cite{Gong95}, one can take $\Re(s)$ and $\Re(w)$ large enough so that
the series in \eqref{e:spectralzeta_def} converges absolutely.
Going further, it is proved in \cite{Gong95} that for $\Re(s)>1$ and $|k|-s\notin\mathbb{Z}_{\geq 0}$, the spectral zeta function $\zeta(w,s)$ possesses meromorphic continuation to the whole $w-$plane and is holomorphic at $w=0$.  With this result,
the zeta regularized determinant $\det\left(\Delta_{2k}-s(1-s)\right)$ for $\Re(s)>1$ and $|k|-s\notin\mathbb{Z}_{\geq 0}$ is defined as
\begin{equation}\label{e:det_def}
\det\left(\Delta_{2k}-s(1-s)\right) = e^{-\frac{\partial}{\partial w}\left.\zeta(w, s)\right\vert_{w=0}}.
\end{equation}
With this, \cite[Theorem 3]{Gong95} proves a relation between the zeta regularized product \eqref{e:det_def}
and the Selberg zeta function $Z(s)=Z(s;\chi)$.  Indeed, for $\Re(s)>1$ and $|k|-s\notin\mathbb{Z}_{\geq 0}$ one has that
\begin{equation}\label{e:det_Delta}
\det\left(\Delta_{2k}-s(1-s)\right)
= Z(s) Z_{I}(s) e^{\tilde{c}},
\end{equation}
for an explicitly computed function $ Z_{I}(s)$ and a constant $\tilde c$ which are given below. Namely, for $s\in \mathbb{C}\setminus(-\infty, |k|]$, it is shown that
\begin{align}\label{e:ZI_def}\nonumber
Z_{I}(s) = \exp\bigg\{&\frac{m\omega(\mathcal{F})}{2\pi} \bigg(s\log(2\pi) + s(1-s) + \left(\frac{1}{2}+k\right) \log\Gamma\left(s+k\right)
\\ &+ \left(\frac{1}{2}-k\right) \log \Gamma\left(s-k\right) - \log G\left(s+k+1\right)-\log G\left(s-k+1\right)\bigg)\bigg\},
\end{align}
and the constant $\tilde c$ is expressed as
\begin{align}\label{eq. c tilde}
\widetilde{c}= -m\frac{{\rm vol}(X)}{2\pi}\left(\log\sqrt{2\pi} + \frac{1}{4}-2\zeta'(-1)\right),
\end{align}
with $\zeta$ signifying the Riemann zeta function.

Let us emphasize that in all formulas above, and in the sequel below, we use the principal branch of the logarithm with its argument in
the range $(-\pi , \pi]$ in order to define rational powers of functions.

\section{Distribution of pseudo primes}\label{sec:pseudo_primes}

We will follow the notation as above, so $\Gamma$ is a Fuchsian group of the first kind and $\chi$ an $m\times m$ unitary multiplier system on $\Gamma$.
Let $P$ be a hyperbolic class $P$ in $\Gamma$ with $P=P_{0}^{n}$ for some primitive class $P_{0}$.  Set
$$
\Lambda(P) := \frac{\log N(P_{0})}{1-N(P)^{-1}}.
$$
In this section, we derive an explicit asymptotic formula for the function
$$
\Psi(x,\chi)= \sum_{P\in\Gamma,\, {\rm Tr}(P) >2: \, N(P)\leq x} \Lambda(P)\mathrm{Tr}(\chi(P)),
$$
as $x$ tends to infinity.
We use the notation that $k(1-k)=\lambda_0\leq \lambda_1\leq \ldots$ is the sequence of eigenvalues of the weighted Laplacian $\Delta_{2k}$ with $\lambda_j=s_j(1-s_j)$ for $k\in(0,1)$ and that  $0=\widetilde{\lambda}_0 <\widetilde{\lambda}_1\leq \ldots$ is the sequence of eigenvalues of the Laplacian $\Delta_0$ with $\widetilde{s}_0=1$ and $\lambda_j=\widetilde{s}_j(1-\widetilde{s}_j)$, for $j\geq 1$. We denote by $\mathcal N_k$ the number of eigenvalues of $\Delta_{2k}$ less than or equal to $1/4$, counted with multiplicities.

The main result of this section is the following theorem.

\begin{theorem}\label{th: PGT1} For $x\geq 2$, we have that
$$
  \Psi(x,\chi)= \sum_{\lambda_j\leq 1/4}m(\lambda_j)\frac{x^{s_j}}{s_j} + g(X,\chi)x^{3/4}
$$
where $m(\lambda_j)$ denotes the multiplicity of the small eigenvalue $\lambda_j$ and $$|g(X,\chi)|\leq C_1\frac{m(g-1)}{{\rm sys}(X)} + (\mathcal N_k -1) \left(\frac{C_2}{1-s_1}+C_3 \right) + (\mathcal N_0 -1) \left(\frac{C_4}{1-\widetilde s_1}+C_5 \right),
$$
for some absolute constants $C_1,\, C_2,\, C_3, \, C_4, \, C_5$ independent of $X$ and $\chi$.
\end{theorem}

\begin{proof}
We follow \cite{FJK11} (see also \cite{Av21}), which is based on refining the arguments from Randol, described in Chapter XI of \cite{Ch84}. The proof consists of two parts.
First, using the trace formula with the test function $h(r)= e^{-tr^2}$ and $g(u)=\frac{1}{\sqrt{4\pi t}}e^{-\frac{u^2}{4t}}$ for $t>0$, we
deduce an explicit bound for the implied constant in the Weyl's law, meaning the asymptotic count of the eigenvalues of the weighted
Laplacian $\Delta_{2k}$.
Then, using Randol's family of test functions $h_T^\varepsilon(r)$ for large $T>0$ and small $\varepsilon>0$, as defined below,
and its corresponding Fourier transform $g_T^\varepsilon(u)$, when combined with the bound for the aforementioned Weyl's law, we complete the
proof of the theorem.
\medskip

\textbf{Part 1.} Let us substitute  $h(r)= e^{-tr^2}$ and $g(u)=\frac{1}{\sqrt{4\pi t}}e^{-\frac{u^2}{4t}}$ for $t>0$ in the trace formula and
estimate the geometric side.  The contribution from the \emph{identity term} is bounded by
\begin{align*}
\left|\frac{m\omega(\mathcal{F})}{4\pi}\int\limits_{-\infty}^{\infty}re^{-tr^2}\frac{\sinh(2\pi r)}{\cosh(2\pi r) + \cos (2\pi k)}dr\right| &\leq \frac{m\omega(\mathcal{F})}{2\pi}\int\limits_{0}^{\infty}re^{-tr^2}\frac{\sinh(2\pi r)}{\cosh(2\pi r) -1}dr\\&\leq \frac{m\omega(\mathcal{F})}{2\pi}\left(\int\limits_{0}^{1}r\coth(\pi r)dr + 2 \int\limits_{0}^{\infty}re^{-tr^2} \right)\\
&\leq \frac{m\omega(\mathcal{F})}{2\pi}\left(A_1+\frac{1}{t}\right),
\end{align*}
where $A_1= \int\limits_{0}^{1}r\coth(\pi r)dr$ is an absolute constant.  Numerical methods show that $A_{1} < 1$ and, indeed, can be expressed as a
special value of the logarithm and dilogarithm functions.

The term \eqref{eq:stf_id2} is non-zero for $k\in(1/2,1)$, in which case it equals
$$
\frac{m\omega(\mathcal{F})}{2\pi}e^{t(|k|-1/2)^2}.
$$
Therefore, the identity contribution is bounded by
$$
A_{{\rm Id}, \Gamma}:=\frac{m\omega(\mathcal{F})}{2\pi}\left(A_1+\frac{1}{t} + \delta_{k\in(1/2,1)} e^{t(|k|-1/2)^2}\right),
$$
where $\delta_{k\in(1/2,1)} $ is zero unless $k\in(1/2,1)$, in which case it equals $1$.
\medskip

The \emph{hyperbolic} contribution is estimated by using the following (crude) count of the hyperbolic elements of $\Gamma$. Namely, after using the trivial bound $|{\rm Tr}(\chi(P))|\leq m$, we have that
$$
\left|\sum\limits_{P, \text{\rm Tr}(P)>2}\frac{\text{\rm Tr}(\chi(P))\log N(P_{0})}{N(P)^{1/2}-N(P)^{-1/2}}g(\log N(P))\right|\leq m \sum\limits_{P, \text{\rm Tr}(P)>2}\frac{\log N(P_{0})}{N(P)^{1/2}-N(P)^{-1/2}}g(\log N(P)).
$$
We will estimate the sum on the right-hand side using Theorems 4.1.6 and 6.6.4 from \cite{Buser92}. The argument is as follows.

By Theorem 6.6.4 in \cite{Buser92}, there are at most $(g-1)e^{L+6} $ oriented closed geodesics of length $\leq L$
which are not iterates of closed geodesic of length $\leq 2 {\rm arcsinh}\, 1 $.
The number of closed geodesics of length  $\leq L$ which \textit{are}  iterates of a primitive closed geodesic of length $\leq 2 {\rm arcsinh} \, 1 $ can be
bounded by first computing of the number of  primitive closed geodesic of length $\leq 2 {\rm arcsinh} \, 1 $ and multiplying this number by $\frac{L}{{\rm sys}(X)}$,
where ${\rm sys}(X)$ is the length of the systole, i.e. the shortest closed geodesic on $X$.
By Theorem 4.1.6 in \cite{Buser92}, the number of primitive closed geodesic of length $\leq 2 {\rm arcsinh} \, 1 $ is bounded by $3g-3$;
each such close geodesic has length greater than or equal to  ${\rm sys}(X)$.
Therefore, the number of closed geodesics of length  $\leq L$, which \textit{are}  iterates of  primitive closed geodesic of length $\leq 2 {\rm arcsinh} \, 1 $ is bounded by $(3g-3)\frac{L}{{\rm sys}(X)} $.
In sum, we have that the number of closed geodesics $\mathcal N(L)$ on $X$ with length $\leq L$ satisfies the (crude) bound
\begin{equation}\label{eq. crude bound number geod}
\mathcal N(L)\leq (g-1)\left(e^{L+6} + 3\frac{L}{{\rm sys}(X)}\right).
\end{equation}

To continue, recall that the norm of the hyperbolic element $P$ and the length of the corresponding geodesic $\gamma$ are related by $N(P)=e^{\ell(\gamma)}$.
Since,
$$
\frac{\log N(P_{0})}{N(P)^{1/2}-N(P)^{-1/2}}\leq 1,
$$
we get that
$$
\sum\limits_{P, \text{\rm Tr}(P)>2}\frac{\log N(P_{0})}{N(P)^{1/2}-N(P)^{-1/2}}g(\log N(P))\leq \frac{1}{\sqrt{4\pi t}}\sum\limits_{\gamma} e^{-\frac{\ell(\gamma)^2}{4t}},
$$
where the sum on the right-hand side is taken over all closed geodesics $\gamma$.
Choose $\delta>0$ so that $0 < \delta < {\rm sys}(X)$. Then, by using Stieltjes integration
and \eqref{eq. crude bound number geod}, we have that
\begin{align*}
  \sum\limits_{P, \text{\rm tr}(P)>2}\frac{\log N(P_{0})g(\log N(P)) }{N(P)^{1/2}-N(P)^{-1/2}} &\leq \frac{1}{2t\sqrt{4\pi t}}\int\limits_{{\rm sys}(X)-\delta}^{\infty} x \mathcal N(x) e^{-\frac{x^2}{4t}}dx \\
  & \leq \frac{g-1}{2t\sqrt{4\pi t}}\left(e^6\int\limits_{{\rm sys}(X)-\delta}^{\infty} x e^{-\frac{x^2}{4t} + x}dx + \frac{3}{{\rm sys}(X)} \int\limits_{{\rm sys}(X)-\delta}^{\infty} x^2 e^{-\frac{x^2}{4t}}dx \right)  \\
  & \leq \frac{(g-1)(e^6+3)}{2t\sqrt{4\pi t}({\rm sys}(X)-\delta)}\int\limits_{-\infty}^{\infty} x^2 e^{-\frac{x^2}{4t}+x}dx.\\
  &=  \frac{(g-1)(e^6+3)}{{\rm sys}(X)-\delta}(1+2t)e^t.
\end{align*}
The last equality follows from \cite{GR07}, equation 3.462.8 with $\mu=\frac{1}{4t}>0$ and $\nu=\frac{1}{2}$.

From the trace formula and the above
bounds for the identity and hyperbolic contributions, and by taking $\delta= {\rm sys}(X)/2$, we get for 
$t$ in the range $0 < t \leq 4$ that
\begin{align}\label{eq:Weyl_bound}\nonumber
\sum\limits_{n\geq 0}e^{-tr_n^2} &\leq m(2g-2)\left( A_1+\frac{1}{t}+\delta_{k\in(1/2,1)} e^{t(|k|-1/2)^2}\right) +  \frac{2m(g-1)(e^6+3)}{{\rm sys}(X)}(1+2t)e^t\\& \leq  \frac{mB(g-1)}{{\rm sys}(X)}\frac{1}{t},
\end{align}
for some constant $B$ which is explicitly computable and independent of $X$.

Let us now proceed as in \cite{FJK11}.  For any $\lambda >0$, set
$$
\mathbf{N}_k(\lambda)=\# \{\lambda_j: \lambda_j\leq \lambda\}.
$$
Then for $t\in(0,4]$ we have that
$$
\int\limits_0^\infty e^{-t\lambda} d\mathbf{N}_k(\lambda) =e^{-\frac{t}{4}}\sum\limits_{n\geq 0}e^{-tr_n^2} \leq \frac{2mB(g-1)}{{\rm sys}(X)}\frac{1}{t}.
$$
Using Karamata's Tauberian theorem we deduce that
\begin{equation}\label{eq. number eigenv trivial bound 1}
\mathbf{N}_k(\lambda)\leq \frac{8mB(g-1)}{{\rm sys}(X)}\lambda;
\end{equation}
for all $\lambda > 0$; see, for example, Lemma 1.1 of \cite{FJK11}.
Finally, let us note that a better bound for the number of small eigenvalues when $k=0$ is proved in \cite{OR09}.
Namely, it is shown in \cite{OR09} that
\begin{equation}\label{eq. number eigenv trivial bound}
\mathbf{N}_0(1/4)\leq 2g-2.
\end{equation}

\medskip

\textbf{Part 2.} Let us start with the smooth non-negative function
$$
\varphi(x)=\frac{1}{c_0}\left\{
                                    \begin{array}{ll}
                                      \exp\left(\frac{1}{x^2-1}\right), & |x|<1 \\
                                      0, & |x|\geq 1,
                                    \end{array}
                                  \right.
$$
where $c_0=\int\limits_{-1}^1 \exp\left(\frac{1}{x^2-1}\right) dx >0$ is the normalizing constant
so then the integral of $\varphi(x)$ over $\mathbb R$
equals $1$.
For any $\varepsilon>0$ set $\varphi_\varepsilon (x):=\frac{1}{\varepsilon}\varphi(x/\varepsilon)$.
For any $T>0$, let $I_{T}$ be the characteristic function of the interval $[-T,T]$.
For $x\in\mathbb R$, define the function
$$
g_T^\varepsilon(x):=(2 \cosh(x/2))(I_T\ast\varphi_\varepsilon)(x),
$$
where
$$
(I_T\ast\varphi_\varepsilon)(x):=\int\limits_{-\infty}^{\infty}I_T(x-y)\varphi_\varepsilon(y)dy
$$
is the convolution of $I_T$  with $\varphi_\varepsilon$.
The corresponding function $h_T^\varepsilon(y)$, which is the inverse Fourier transform of $g_T^\varepsilon$,
is then given by
\begin{align}
h_T^\varepsilon(y) =2\pi \left( \frac{\sin(T(y-i/2))}{(y-i/2)}\widehat{\varphi_\varepsilon}(y-i/2) + \frac{\sin(T(y+i/2))}{(y+i/2)}\widehat{\varphi_\varepsilon}(y+i/2)\right) ;\label{eq. h eps defn}
\end{align}
see Section 4.2 of \cite{FJK11}.
From Lemma 4.2 of \cite{FJK11}, we have for all $r\geq 0$ and all $0<\varepsilon \leq 1$,
the bound
\begin{equation}\label{eq. bound h-eps}
  |h_T^\varepsilon(r)|\leq ce^{T/2}(1+r)^{-1}(1+\varepsilon r)^{-2}
\end{equation}
for some universal constant $c$. From \eqref{eq. bound h-eps} and \eqref{eq. h eps defn}
we see that $h_T^\varepsilon$ satisfies the required assumptions posed for a test function to be used in the trace formula.
In doing so, we obtain the following identity that
\begin{align}\label{eq:stf2_id1}
\sum\limits_{\lambda_n\leq 1/4} h_T^\varepsilon(r_{n})+ \sum\limits_{\lambda_n>1/4} h_T^\varepsilon(r_{n}) &= \frac{m\omega(\mathcal{F})}{4\pi}\int\limits_{-\infty}^{\infty}rh_T^\varepsilon(r)\frac{\sinh(2\pi r)}{\cosh(2\pi r) + \cos (2\pi k)}dr
\\& \label{eq:stf2_id2} +\frac{m\omega(\mathcal{F})}{2\pi}\sum\limits_{0 \leq l < \vert k \vert -1/2}(\vert k \vert - l - 1/2)h_T^\varepsilon(i(\vert k \vert -l - 1/2))
\\&\label{eq:stf2_hyp1} +\sum\limits_{P, \text{\rm Tr}(P)>2}\text{\rm Tr}(\chi(P))\frac{\log N(P_{0})}{N(P)^{1/2}-N(P)^{-1/2}}g_T^\varepsilon(\log N(P)).
\end{align}

\medskip
\textit{From this point on, we take $T=\log x$ for some $x>{\rm sys}(X)$ and set $\varepsilon=e^{-T/4}$}.
\medskip

Let us study the various terms in \eqref{eq:stf2_id1} through \eqref{eq:stf2_hyp1}, and we
begin with the left-hand side of \eqref{eq:stf2_id1}. Let $\mathcal N_k= \mathcal N(X,k)$ denote the number of small eigenvalues $\lambda_j \leq 1/4$ of $\Delta_k$, counted with their multiplicities; recall that $\Delta_0$ is the Laplacian on $X$
which acts on smooth functions. We let $k(1-k)<\lambda_1=s_1(1-s_1)$. Lemma 4.4 of \cite{FJK11} holds in our case and yields that
\begin{equation}\label{eq. intermediate bound}
\left|\sum\limits_{\lambda_n\leq 1/4} h_T^\varepsilon(r_{n})-\sum\limits_{\lambda_n\leq 1/4} m(\lambda_n) \frac{e^{s_nT}}{s_n}\right| \leq \left( c_1+ (\mathcal N_k -1) \left(\frac{c_2}{1-s_1}+c_3\right)\right)e^{3T/4},
\end{equation}
where constants $c_1,\,c_2,\,c_3$ are absolute.
The sum over eigenvalues $\lambda_n>1/4$ in \eqref{eq:stf2_id1} is estimated using \eqref{eq. number eigenv trivial bound 1}. Specifically,
let us write
$$
\sum\limits_{\lambda_n > 1/4} h_T^\varepsilon(r_{n})=\int\limits_0^\infty h_T^\varepsilon(r) d\mu_k(r),
$$
where
$$
\mu_k(r)=\# \{r_n:0\leq r_n\leq r\}=\mathbf{N}_k(r^2+1/4).
$$
From \eqref{eq. number eigenv trivial bound 1}, we have that
$$
\mathbf{N}_k(r^2+1/4)\leq \frac{8mB(g-1)}{{\rm sys}(X)}(r^2+1/4).
$$
Using \eqref{eq. bound h-eps} and proceeding as in the proof of Lemma 4.3 of \cite{FJK11}, we immediately deduce that
$$
\left|\sum\limits_{\lambda_n > 1/4} h_T^\varepsilon(r_{n})\right|\leq \frac{mb(g-1)}{{\rm sys}(X)}e^{3T/4},
$$
for some absolute constant $b$. This, combined with \eqref{eq. intermediate bound} yields a bound for both terms in
the spectral side of the trace formula \eqref{eq:stf2_id1},
namely that
\begin{align}\label{eq. spect side est}\nonumber
&\left|\sum\limits_{\lambda_n \leq 1/4} h_T^\varepsilon(r_{n}) + \sum\limits_{\lambda_n > 1/4} h_T^\varepsilon(r_{n})- \sum\limits_{\lambda_n\leq 1/4} m(\lambda_n) \frac{e^{s_nT}}{s_n}\right|  \\ & \hskip 10mm \leq
\left(\frac{mb(g-1)}{{\rm sys}(X)} + c_1+ (\mathcal N_k -1) \left(\frac{c_2}{1-s_1}+c_3\right) \right)e^{3T/4}.
\end{align}

Now, we bound the identity contribution, which is the right-hand-side of  \eqref{eq:stf2_id1} together with \eqref{eq:stf2_id2}.
Trivially, the contribution \eqref{eq:stf2_id2}, if non-zero, is bounded by an absolute constant multiplied with $m(2g-2)$. The contribution of the
right-hand-side of \eqref{eq:stf2_id1} can be estimated using \eqref{eq. bound h-eps}, thus giving
\begin{align}\nonumber
\left|\int\limits_{-\infty}^{\infty}h_T^\varepsilon(r)\frac{r\sinh(2\pi r)}{\cosh(2\pi r) + \cos (2\pi k)}dr\right| &\leq 2 ce^{T/2} \bigg(\int\limits_0^1\frac{r \sinh(2\pi r)}{\cosh(2\pi r) -1}dr + \int\limits_1^\infty \frac{1}{(1+\varepsilon r)^2}\frac{\sinh(2\pi r)}{\cosh(2\pi r) -1}dr\bigg)\\ &\leq 2 ce^{T/2}\left(c_4 + \frac{2}{\epsilon} \int\limits_0^\infty \frac{dy}{(1+y)^2}\right) \nonumber\\
&\leq c_5 e^{3T/4},
\label{eq. identity bound}
\end{align}
for some absolute constant $c_5$.  When combining, we have that the contribution from
identity terms of \eqref{eq:stf2_id1} and \eqref{eq:stf2_id2}  is bounded by $mc_6(g-1)e^{3T/4}$ for some absolute constant $c_6$.

It is remains to analyze the hyperbolic contribution \eqref{eq:stf2_hyp1}.  We will conduct the analysis in two steps.
First, assume that $k=0$, $m=1$, in which case $\chi={\rm Id}_1$ is the identity operator and $\text{\rm Tr}(\chi(P))=1$
for all $P$. In this case, we may proceed analogously as in Section 4.4. of \cite{FJK11}. Set
$$
H_\epsilon(T,{\rm Id}_1):= \sum\limits_{P, \text{\rm Tr}(P)>2}\frac{\log N(P_{0})}{N(P)^{1/2}-N(P)^{-1/2}}g_T^\varepsilon(\log N(P)).
$$
Observe that
\begin{equation}\label{eq. H eps bound 1}
H_\varepsilon(T-\varepsilon,{\rm Id}_1)\leq H_\epsilon(T,{\rm Id}_1) \leq H_\varepsilon(T+\varepsilon,{\rm Id}_1)
\end{equation}
and
\begin{equation}\label{eq. H eps bound 2}
H_\varepsilon(T-\varepsilon,{\rm Id}_1)\leq \Psi(e^T,{\rm Id}_1) + \sum_{N(P)\leq e^T} \frac{\log N(P_{0}) N(P)^{-1}}{1-N(P)^{-1}} \leq H_\varepsilon(T+\varepsilon,{\rm Id}_1).
\end{equation}
We note that the term in the middle was denoted in \cite{FJK11} by $H(T)$, and the above inequality actually follows from
arguments in Chapter XI of \cite{Ch84}).

From the trace formula \eqref{eq:stf2_id1}--\eqref{eq:stf2_hyp1} and bounds \eqref{eq. intermediate bound}--\eqref{eq. identity bound} with $\chi={\rm Id}_1$ we deduce that
\begin{align}\label{eq. H eps bound 3}\nonumber
&\left|H_\varepsilon(T\pm\varepsilon,{\rm Id}_1)  - \sum\limits_{\widetilde\lambda_n\leq 1/4} m(\widetilde\lambda_n) \frac{e^{\widetilde s_n(T\pm\varepsilon)}}{\widetilde s_n}\right|\\& \hskip 10mm \leq \left(\frac{c_7(g-1)}{{\rm sys}(X)} + (\mathcal N_0 -1) \left(\frac{c_2}{1-\widetilde s_1}+c_3\right) \right)e^{3(T\pm \varepsilon)/4},
\end{align}
for some absolute constant $c_7$. Therefore,
$$
\left|H_\varepsilon(T+\varepsilon,{\rm Id}_1)  - H_\varepsilon(T-\varepsilon,{\rm Id}_1)\right|\leq  \left(\frac{c_8(g-1)}{{\rm sys}(X)} + (\mathcal N_0 -1) \left(\frac{c_9}{1-\widetilde s_1}+c_{10}\right) \right)e^{3T/4}
$$
for some absolute constants $c_8,\, c_9,\, c_{10}$. From \eqref{eq. H eps bound 1}--\eqref{eq. H eps bound 3} we deduce that
\begin{align*}
&\left|\Psi(e^T,{\rm Id}_1) + \sum_{N(P)\leq e^T} \frac{\log N(P_{0}) N(P)^{-1}}{1-N(P)^{-1}} - H_\varepsilon(T,{\rm Id}_1) \right|
\\ & \hskip 10mm \leq \left(\frac{c_8(g-1)}{{\rm sys}(X)} + (\mathcal N_0 -1) \left(\frac{c_9}{1-\widetilde s_1}+c_{10}\right) \right)e^{3T/4}.
\end{align*}
Therefore,
\begin{align}\label{eq. pre-final for id}\nonumber
\left|\Psi(e^T,{\rm Id}_1) -\sum\limits_{\widetilde \lambda_n\leq 1/4} m(\widetilde \lambda_n) \frac{e^{\widetilde s_n T}}{\widetilde s_n}  \right|& \leq  \left(\frac{c_8(g-1)}{{\rm sys}(X)} + (\mathcal N_0 -1) \left(\frac{c_9}{1-\widetilde s_1}+c_{10}\right) \right)e^{3T/4} \\&+ \sum_{N(P)\leq e^T} \frac{\log N(P_{0}) N(P)^{-1}}{1-N(P)^{-1}}.
\end{align}
Finally, what is left is to estimate the last sum on the right-hand side of \eqref{eq. pre-final for id}. We do so
using the crude bound \eqref{eq. crude bound number geod}, which gives that
\begin{align}\label{eq. sum with -1}
\underset{N(P)\leq e^T}{\sum_{P, \mathrm{tr}(P)>2}} \frac{\log N(P_{0})N(P)^{-1}}{1-N(P)^{-1}}
&\leq  \underset{\ell(\gamma)\leq T}{\sum_{\gamma, \ell(\gamma)\leq T}} \frac{\ell(\gamma_0)e^{-\ell(\gamma)}}{1-e^{-\mathrm{sys}(X)}} \\
&\leq \frac{1}{1-e^{-\mathrm{sys}(X)}} \int_{\mathrm{sys}(X) -\delta}^{T} x e^{-x}d \mathcal N(x) \nonumber \\
&\leq (g-1) \left(c_{11} T^2 + \frac{c_{12}}{{\rm sys}(X)}T\right) \nonumber\\&\leq (g-1)c_{13}\left(1+\frac{1}{{\rm sys}(X)} \right)e^{3T/4},\nonumber
\end{align}
for some absolute constant $c_{13}$.
Thus, Theorem \ref{th: PGT1} follows, in the case when $\chi={\rm Id}_m$,
by \eqref{eq. pre-final for id} and \eqref{eq. sum with -1}.

Let us now consider Theorem \ref{th: PGT1} for an arbitrary $\chi$, meaning any admissible $k\in(0,1)$.
We have that the hyperbolic contribution to the trace formula equals
\begin{align*}
\sum\limits_{P, \text{\rm Tr}(P)>2}&\text{\rm Tr}(\chi(P))\frac{\log N(P_{0})}{N(P)^{1/2}-N(P)^{-1/2}}g_T^\varepsilon(\log N(P))\\&=\sum\limits_{P, \text{\rm Tr}(P)>2}\text{\rm Tr}(\chi(P))\frac{\log N(P_{0})(1+N(P)^{-1})}{1-N(P)^{-1}} (I_T\ast\varphi_\varepsilon)(\log N(P)).
\end{align*}
Observe that
$$
(I_T\ast\varphi_\varepsilon)(x)-I_T(x)=0
$$
unless $x\in[-T-\varepsilon, -T+\varepsilon] \cup [T-\varepsilon, T+\varepsilon]$ in which case $\vert I_T\ast\varphi_\varepsilon)(x)-I_T(x)\vert \leq 2$.
For $T \geq 1$ we have, using \eqref{eq. sum with -1}, that
\begin{multline}\label{eq. Hyp cont pre final}
\bigg|\sum\limits_{P, \text{\rm Tr}(P)>2, N(P)\leq e^T}\text{\rm Tr}(\chi(P))\frac{\log N(P_{0})(1+N(P)^{-1})}{1-N(P)^{-1}} \\ - \sum\limits_{P, \text{\rm Tr}(P)>2}\text{\rm Tr}(\chi(P))\frac{\log N(P_{0})(1+N(P)^{-1})}{1-N(P)^{-1}} (I_T\ast\varphi_\varepsilon)(\log N(P))\bigg| \\ \leq 2m \underset{e^{T-\varepsilon}\leq N(P)\leq e^{T+\varepsilon}}{\sum_{P, \text{\rm Tr}(P)>2}}\frac{\log N(P_{0})(1+N(P)^{-1})}{1-N(P)^{-1}} \\ \leq 2m(\Psi(e^{T+\varepsilon},{\rm Id}_1)- \Psi(e^{T-\varepsilon},{\rm Id}_1)) + m(g-1)c_{13}\left(1+\frac{1}{{\rm sys}(X)} \right)e^{3T/4}.
\end{multline}
Note that $ -T+\varepsilon=  -T+e^{-T/4}<0$, which implies that the sum over $P$ with $e^{-T-\varepsilon}\leq N(P)\leq e^{-T+\varepsilon}$ is empty.

From above, Theorem \ref{th: PGT1} holds for $\Psi(e^{T\pm\varepsilon},{\rm Id}_1)$.  Hence,
\begin{align}\label{eq:id_bound_one}\nonumber
\Psi(e^{T+\varepsilon},{\rm Id}_1) &- \Psi(e^{T-\varepsilon},{\rm Id}_1)
\\ &\leq 2\varepsilon\sum\limits_{\widetilde \lambda_n\leq 1/4} m(\widetilde \lambda_n) e^{\widetilde s_n (T-\varepsilon) }+ \left(\frac{c_{14}(g-1)}{{\rm sys}(X)} + (\mathcal N_0 -1) \left(\frac{c_{15}}{1-\widetilde s_1}+c_{16}\right) \right)e^{3T/4}.
\end{align}
Note that $s_n\leq 1$, and recall the bound \eqref{eq. number eigenv trivial bound} for the number of small eigenvalues.
Since $\varepsilon=e^{-T/4}$, \eqref{eq:id_bound_one} becomes the bound
$$
\Psi(e^{T+\varepsilon},{\rm Id}_1)- \Psi(e^{T-\varepsilon},{\rm Id}_1)\leq \left(\frac{c_{14}(g-1)}{{\rm sys}(X)} + (\mathcal N_0 -1) \left(\frac{c_{15}}{1-\widetilde s_1}+c_{17}\right) \right)e^{3T/4},
$$
for some absolute constants $c_{14}$, $c_{15}$ and $c_{17}$. By inserting this inequality into \eqref{eq. Hyp cont pre final} we conclude that
\begin{multline*}
\bigg| \Psi(e^T,\chi) - \sum\limits_{P, \text{\rm tr}(P)>2}\text{\rm tr}(\chi(P))
\frac{\log N(P_{0})(1+N(P)^{-1})}{1-N(P)^{-1}}g_T^\varepsilon(\log N(P))\bigg| \\ \hskip 5mm
\leq \left(\frac{mc_{14}(g-1)}{{\rm sys}(X)} + (\mathcal N_0 -1) \left(\frac{c_{15}}{1-\widetilde s_1}+c_{17}\right) \right)e^{3T/4}
+\sum_{N(P)\leq e^T} m\frac{\log N(P_{0}) N(P)^{-1}}{1-N(P)^{-1}} .
\end{multline*}
Then, combining with \eqref{eq. sum with -1}, we have that
\begin{align}\label{eq:id_bound_two}
&\bigg| \Psi(e^T,\chi) - \sum\limits_{P, \text{\rm tr}(P)>2}\text{\rm tr}(\chi(P))\frac{\log N(P_{0})(1+N(P)^{-1})}{1-N(P)^{-1}}g_T^\varepsilon(\log N(P))\bigg| \\ & \hskip 7mm \leq
\left(\frac{mc_{18}(g-1)}{{\rm sys}(X)} + (\mathcal N_0 -1) \left(\frac{c_{15}}{1-\widetilde s_1}+c_{17}\right) \right)e^{3T/4} .
\end{align}
Next, we use the bounds \eqref{eq. spect side est} for the spectral contribution and \eqref{eq. identity bound} for the identity contribution
in the trace formula \eqref{eq:stf2_id1}--\eqref{eq:stf2_hyp1} so then \eqref{eq:id_bound_two} reduces to
\begin{equation}\label{eq. penultimate bound}
\left| \Psi(e^T,\chi) - \sum\limits_{\lambda_n\leq 1/4}   m(\lambda_n) \frac{e^{s_nT}}{s_n}\right| \leq C(X,\chi)e^{3T/4},
\end{equation}
where
$$
C(X,\chi) = \left( \frac{mc_{19}(g-1)}{{\rm sys}(X)} + (\mathcal N_0 -1) \left(\frac{c_{15}}{1-\widetilde s_1}+c_{17}\right) + (\mathcal N_k -1) \left(\frac{c_2}{1-s_1}+c_3 \right) \right)
$$
for some absolute constant $c_{19}$.
Using the notation that $x=e^T$ in \eqref{eq. penultimate bound}, we finally get that
$$
\left|\Psi(x, \chi)-  \sum\limits_{\lambda_n\leq 1/4}m(\lambda_n) \frac{x^{s_n}}{s_n} \right| \leq C(X,\chi) x^{3/4},
$$
where the constant of $C(X,\chi)$ depends on the following: genus $g$, dimension $m$, the number of small eigenvalues $\mathcal N_0$ of $\Delta_0$ counted with
multiplicities, and the number of small eigenvalues $\mathcal N_k$ of $\Delta_{2k}$ counted with multiplicities.  Additionally, the dependence on the
the first eigenvalues $\Delta_0$ and $\Delta_{2k}$ is given explicitly in the statement of the theorem.

With all this, the proof of Theorem \ref{th: PGT1} is complete.
\end{proof}

\section{A uniform prime geodesic theorem}\label{sec:uniform_PGT}

Following the approach established in the proof of Theorem 2 from \cite{WX25}, which is motivated by \cite{BP22},
it is possible to prove a uniform version of Theorem \ref{th: PGT1} with an error term that has worse bound in $T$ but
depends solely on the geometry of the underlying surface.  The result we obtain is the following.

\begin{theorem} \label{th. PGT2}  With the notation as above, there is a universal constant $x_{0}$
such that for $x\geq x_{0}$,
$$
  \Psi(x,\chi)= \sum_{\lambda_j\leq 1/4}m(\lambda_j)\frac{x^{s_j}}{s_j} + \mathcal G(x; X,\chi)
$$
where $m(\lambda_j)$ denotes the multiplicity of the small eigenvalue $\lambda_j$ and
$$
|\mathcal G(x; X,\chi)|\leq \mathcal C_1mg x^{5/6} + \mathcal C_2 mg x^{2/3} \max\left\{ 0, \log\left(\frac{1}{{\rm sys}(X) x^{1/6}}\right)\right\},
$$
for some absolute constants $\mathcal C_1,\, \mathcal C_2$ independent of $X$ and $\chi$.
\end{theorem}

\begin{proof}
Our proof follows the method which proves \cite[Theorem 2]{WX25}, with some minor modifications due to two facts.
First, we are counting $\Lambda(P)\mathrm{Tr}(\chi(P))$ over all hyperbolic elements $P$ with norm $\leq x$ instead
of counting $1$ over all primitive geodesics; and second, as a consequence, the corresponding counting function is not necessarily increasing.

We will begin by describing how to prove Theorem \ref{th. PGT2} when $\chi$ is the identity one-dimensional representation.  Afterwards,
we will use this result to complete the proof in general case. As in the proof of Theorem \ref{th: PGT1}, out starting point
is the trace formula with $k=0$ and $m=1$.

Let $\eta(x)$ be a non-negative, smooth, even function with support $\text{supp}(\eta) \subset (-\frac{1}{2}, \frac{1}{2})$, total integral
$\int_{\mathbb{R}} \eta(x)dx = 1$, and with $\eta(0) = \max_{x \in \mathbb{R}} \eta(x) = 2$. Set $\eta_{\epsilon}(x) = \frac{1}{\epsilon} \eta(\frac{x}{\epsilon})$. For any $0 < \epsilon < 0.01$ and $L > 1$, we define, as in \cite{WX25},
\begin{equation*}
    f_{\epsilon}(x) = (\eta_{\epsilon} * \eta_{\epsilon})(x).
\end{equation*}
Also, define the functions $\varphi_{L,\epsilon}^{+}$ and $\varphi_{L,\epsilon}^{-}$ by
\begin{equation}\label{eq. start for count WX}
    \varphi_{L,\epsilon}^{\pm}(x) = \frac{1}{2}(f_{\epsilon}(x - L) + f_{\epsilon}(x + L)) \pm f_{\epsilon}(x).
\end{equation}
We shall now prove upper and lower bounds for $ \varphi_{L,\epsilon}^{\pm}(x)$ and their Fourier transforms. As in \cite[Lemma 23]{WX25}, it follows that
$$
\Psi(L,{\rm Id}_1)=2\int_{1-\epsilon}^{L+\epsilon} \sum_{k=1}^\infty\sum_{1<k\ell(\gamma)\leq L}\frac{\ell(\gamma)}{1-e^{-k\ell(\gamma)}}  \varphi_{\tau,\epsilon}^{\pm}(k\ell(\gamma))d\tau.
$$
Consider the trace formula with $k=0$ and test function being $g= \varphi_{\tau,\epsilon}^{\pm}$ for $\tau\in[1-\epsilon, L+\epsilon]$.
In doing so, we obtain the identity that
\begin{align}\label{eq. Trace WX}\nonumber
\sum_{k=1}^{\infty} \sum_{\gamma \in\mathcal P(X)} \frac{\ell(\gamma)}{2 \sinh \frac{k\ell(\gamma)}{2}} \varphi_{\tau,\epsilon}^{\pm}(k\ell(\gamma)) &= \sum_{0 \le \widetilde\lambda_j \le 1/4} \widehat{\varphi}_{\tau,\epsilon}^{\pm}(\widetilde r_j) + \sum_{ \widetilde\lambda_j > 1/4} \widehat{\varphi}_{\tau,\epsilon}^{\pm}(\widetilde r_j) \\&- (g - 1) \int_{\mathbb{R}} r \tanh(\pi r) \widehat{\varphi}_{\tau,\epsilon}^{\pm}(r) dr,
\end{align}
where $\mathcal P(X)$ denotes the set of all primitive geodesics and $\widetilde  s_j=1/2 + \widetilde  r_j$.
The right-hand side of \eqref{eq. Trace WX} is estimated in the same way as in \cite{WX25}.  We note that the bounds one obtains are independent of
the fact that we are interested in the \emph{weighted} distribution function.
In doing so, and by following the proof of Lemma 24 of \cite{WX25}, we find estimates for $\varphi_{\tau,\epsilon}^{-}$ as
\begin{equation} \label{eq. bound - large eigenv}
\sum_{\widetilde\lambda_j > \frac{1}{4}} \widehat{\varphi}_{\tau,\epsilon}^{-}(\widetilde  r_j) \le 0.
\end{equation}
and
\begin{equation}\label{eq. bound - small eigenv}
\sum_{0 \le \widetilde\lambda_j \le \frac{1}{4}} \widehat{\varphi}_{\tau,\epsilon}^{-}(\widetilde  r_j) \le \frac{1}{2} \sum_{0 \le \widetilde\lambda_j \le \frac{1}{4}} e^{(\widetilde  s_j - \frac{1}{2})\tau} + C_{1,\eta}\left(g\epsilon^2 e^{\frac{1}{2}\tau}\right),
\end{equation}
for some constant $C_{1,\eta}$ depending only upon $\eta$.
Analogous estimates for $\varphi_{\tau,\epsilon}^{+}$ are derived using the proof of  Lemma 25 of \cite{WX25}), namely that
\begin{equation} \label{eq. bound + large eigenv}
\sum_{\widetilde\lambda_j > \frac{1}{4}} \widehat{\varphi}_{\tau,\epsilon}^{+}(\widetilde r_j) \ge 0,
\end{equation}
and
\begin{equation} \label{eq. bound + small eigenv}
\sum_{0 \le\widetilde \lambda_j \le \frac{1}{4}} \widehat{\varphi}_{\tau,\epsilon}^{+}(\widetilde  r_j) \ge \frac{1}{2} \sum_{0 \le\widetilde \lambda_j \le \frac{1}{4}} e^{(\widetilde  s_j - \frac{1}{2})\tau} - C_{2,\eta}\left(g\epsilon^2 e^{\frac{1}{2}\tau}\right),
\end{equation}
for some constant $C_{2,\eta}$ depending only upon $\eta$.
Furthermore,
\begin{equation} \label{eq. bound CT}
\left|\int_{\mathbb{R}} r \tanh(\pi r) \widehat{\varphi}_{\tau,\epsilon}^{\pm}(r) dr\right|\leq \frac{C_{3,\eta}}{\epsilon^2}.
\end{equation}
Now, we will modify proofs of Lemma 24 and 25 of \cite{WX25} in which the left-hand side of the trace formula \eqref{eq. Trace WX} is analyzed. For $\varphi_{\tau,\epsilon}^{-}$, we use that $\varphi_{\tau,\epsilon}^{-}(x) \ge 0$ when $x > \epsilon$, and $\varphi_{\tau,\epsilon}^{-}(x) = 0$ when $x \ge \tau + \epsilon$. Moreover, we use the lower bound $\varphi_{\tau,\epsilon}^{-}(x) \ge -\frac{2}{\epsilon}$ combined with the fact that $e^{-x/2}$ is decreasing function to get that
\begin{align*}
\sum_{k=1}^{\infty} &\sum_{\gamma \in\mathcal P(X)} \frac{ \ell(\gamma)}{2 \sinh \frac{k\ell(\gamma)}{2}} \varphi_{\tau,\epsilon}^{-}(k\ell(\gamma)) \\ &\geq e^{-\frac{\tau+\epsilon}{2}} \sum_{k=1}^{\infty} \sum_{\epsilon < k\ell (\gamma)<L} \frac{  \ell(\gamma)}{1-e^{-k\ell(\gamma)}} \varphi_{\tau,\epsilon}^{-}(k\ell(\gamma))-\frac{2}{\epsilon}\sum_{k=1}^{\infty} \sum_{k\ell(\gamma) \leq \epsilon } \frac{ \ell(\gamma)}{2 \sinh \frac{k\ell(\gamma)}{2}}\\
&\geq e^{-\frac{\tau+\epsilon}{2}} \sum_{k=1}^{\infty} \sum_{1 < k\ell (\gamma)<L} \frac{ \ell(\gamma)}{1-e^{-k\ell(\gamma)}} \varphi_{\tau,\epsilon}^{-}(k\ell(\gamma))-C \frac{1}{\epsilon}\left(g+  \sum_{\gamma \in\mathcal P(X), \ell(\gamma)<\epsilon}\log\left(\frac{\epsilon}{\ell(\gamma)}\right) \right),
\end{align*}
for some absolute constant $C$.  We note that the second inequality above was taken from equation (12) of \cite{WX25}; we refer to the arxiv preprint version of \cite{WX25}.  Let us now combine the trace formula \eqref{eq. Trace WX} with the bounds \eqref{eq. bound - large eigenv}, \eqref{eq. bound - small eigenv} and \eqref{eq. bound CT}.  In doing so, we get for $\tau\in[1- \epsilon, L+\epsilon]$ that
\begin{multline}\label{eq. upper bound}
\sum_{k=1}^{\infty} \sum_{1 < k\ell (\gamma)<L} \frac{ \ell(\gamma)}{1-e^{-k\ell(\gamma)}} \varphi_{\tau,\epsilon}^{-}(k\ell(\gamma))
 \\ \leq\frac{1}{2}\sum_{0 \le \widetilde \lambda_j \le \frac{1}{4}} e^{\widetilde s_j \tau}  +  C(\eta)\left(  g\epsilon e^{\tau} + \frac{g}{\epsilon^2}e^{\tau/2} + \frac{e^{\tau/2}}{\epsilon} \sum_{\gamma \in\mathcal P(X), \ell(\gamma)<\epsilon}\log\left(\frac{\epsilon}{\ell(\gamma)}\right) \right),
\end{multline}
for some absolute constant $C(\eta)$ depending only upon $\eta$.
By proceeding analogously as is the proof of \cite[Lemma 25]{WX22}, meaning that we use the trace formula \eqref{eq. Trace WX}  combined with the bound
$ 0 \le \varphi_{L,\epsilon}^{+}(x) \le \frac{4}{\epsilon}$ and bounds \eqref{eq. bound + large eigenv}, \eqref{eq. bound + small eigenv} and \eqref{eq. bound CT} we get that
\begin{multline}\label{eq. lower bound}
\sum_{k=1}^{\infty} \sum_{\epsilon < k\ell (\gamma)<L} \frac{ \ell(\gamma)}{1-e^{-k\ell(\gamma)}} \varphi_{\tau,\epsilon}^{+}(k\ell(\gamma)) \\ \geq \frac{1}{2}\sum_{0 \le \widetilde\lambda_j \le \frac{1}{4}} e^{\widetilde s_j \tau}  -B(\eta)\left(  g\epsilon e^{\tau} + \frac{g(1+\epsilon \tau)}{\epsilon^2}e^{\tau/2} + \frac{e^{\tau/2}}{\epsilon} \sum_{\gamma \in\mathcal P(X), \ell(\gamma)<\epsilon}\log\left(\frac{\epsilon}{\ell(\gamma)}\right) \right),
\end{multline}
for some absolute constant $B(\eta)$ depending only upon $\eta$.

The statement of the theorem for $\chi=\mathrm{Id}_V$ now follows by integrating \eqref{eq. upper bound} and \eqref{eq. lower bound}
with respect to $\tau$ from $1-\epsilon$ to $L+\epsilon$, taking $L$ sufficiently large so that $\epsilon=e^{-L/6}<0.001$ and with $x=e^L$.
For further details, we refer to the proof of Theorem 21 from \cite{WX25}.

Let us point out that at this point we can state a value for $x_{0}$.  Indeed, it is necessary
that $x_{0} = e^{L} = (0.001)^{-6} = 10^{18}$.

Now, let $\chi$ be arbitrary unitary multiplier system. According to Lemma 3.3 on p. 473 of \cite{Hejhal83} for any $x\geq 1$ we have that
\begin{equation}\label{eq:hejhal_bound}
\Psi(x,\chi)=\sum_{P\in\Gamma,\, {\rm Tr}(P) >2: \, N(P)\leq x} \Lambda(P)\mathrm{Re}\left(\mathrm{Tr}(\chi(P))\right).
\end{equation}
Therefore, if we fix a real number $c\in(0,1/m)$, the distribution function
\begin{equation}\label{eq:real_series}
\Psi_c(x,\chi):=\Psi(x,{\rm Id}_1)+c\Psi(x,\chi)= \sum_{P\in\Gamma,\, {\rm Tr}(P) >2: \, N(P)\leq x} \Lambda(P)\left(1+c\mathrm{Re}\left(\mathrm{Tr}(\chi(P))\right)\right)
\end{equation}
is non-decreasing.  As a result, we can apply the same methodology as above to study $\Psi_c(x,\chi)$.
Namely, we start with the identity
$$
\Psi_c(L,\chi)=2\int_{1-\epsilon}^{L+\epsilon} \sum_{k=1}^\infty\sum_{1<k\ell(\gamma)\leq L}\frac{\left(1+c\mathrm{Re}\left(\mathrm{Tr}(\chi(\gamma))\right) \right)\ell(\gamma)}{1-e^{-k\ell(\gamma)}}  \varphi_{\tau,\epsilon}^{\pm}(k\ell(\gamma))d\tau.
$$
The corresponding trace formula, for $\tau\in[1-\epsilon, L+\epsilon]$ for $k \in (0,1)$ gives that
\begin{align}\label{eq. Trace general sum}\nonumber
&\sum_{k=1}^{\infty} \sum_{\gamma \in\mathcal P(X)} \frac{\left(1+c\mathrm{Re}\left(\mathrm{Tr}(\chi(\gamma))\right) \right)\ell(\gamma)}{2 \sinh \frac{k\ell(\gamma)}{2}} \varphi_{ \tau,\epsilon}^{\pm}(k\ell(\gamma)) \\& \nonumber=\sum_{0 \le \widetilde \lambda_j \le 1/4} \widehat{\varphi}_{\tau,\epsilon}^{\pm}(\widetilde r_j) + \sum_{ \widetilde\lambda_j > 1/4} \widehat{\varphi}_{\tau,\epsilon}^{\pm}(\widetilde r_j) - (g - 1) \int_{\mathbb{R}} r \tanh(\pi r) \widehat{\varphi}_{\tau,\epsilon}^{\pm}(r) dr\\&\nonumber + c\bigg(\sum_{0 \le \lambda_j \le 1/4} \widehat{\varphi}_{\tau,\epsilon}^{\pm}(r_j) + \sum_{ \lambda_j > 1/4} \widehat{\varphi}_{\tau,\epsilon}^{\pm}( r_j) -m(g-1) \int\limits_{-\infty}^{\infty}r\widehat{\varphi}_{\tau,\epsilon}^{\pm}(r)\frac{\sinh(2\pi r)}{\cosh(2\pi r) + \cos (2\pi k)}dr \\&
-2m(g-1)\sum\limits_{0 \leq l < \vert k \vert -1/2}(\vert k \vert - l - 1/2)\widehat{\varphi}_{\tau,\epsilon}^{\pm}(i(\vert k \vert -l - 1/2))\bigg).
\end{align}
Observe that $ \widehat{\varphi}_{\tau,\epsilon}^{\pm}(r) = (\cos(\tau r) -1) \widehat{\eta}(\epsilon r)^2$.  As a result,
the identity contribution from the weight $k\in(0,1)$ trace formula can be estimated as
$$
\left| \int_{\mathbb{R}} r \tanh(\pi r) \widehat{\varphi}_{\tau,\epsilon}^{\pm}(r) dr\right|\leq  2\int_{\mathbb{R}} r \coth(\pi r) \widehat{\eta}(\epsilon r)^2 dr\leq \frac{1}{\epsilon^2}A(\eta),
$$
where
$$
A(\eta)=\int_{\mathbb{R}} r \coth(\pi r)\widehat{\eta}(r)^2 dr.
$$
Therefore it remains to estimate the left-hand side of \eqref{eq. Trace general sum}. To do so, we proceed analogously as above, now using that $0<1-cm\leq 1+c\mathrm{Re}\left(\mathrm{Tr}(\chi(\gamma))\right)\leq 1+cm$.   This gives that
\begin{align*}
\sum_{k=1}^{\infty} &\sum_{\gamma \in\mathcal P(X)} \frac{ \left(1+c\mathrm{Re}\left(\mathrm{Tr}(\chi(\gamma))\right) \right) \ell(\gamma)}{2 \sinh \frac{k\ell(\gamma)}{2}} \varphi_{L,\epsilon}^{-}(k\ell(\gamma)) \\ &\geq e^{-\frac{\tau+\epsilon}{2}} \sum_{k=1}^{\infty} \sum_{\epsilon < k\ell (\gamma)<L} \frac{  \left(1+c\mathrm{Re}\left(\mathrm{Tr}(\chi(\gamma))\right) \right)\ell(\gamma)}{1-e^{-k\ell(\gamma)}} \varphi_{L,\epsilon}^{-}(k\ell(\gamma))-\frac{2(1+mc)}{\epsilon}\sum_{k=1}^{\infty} \sum_{k\ell(\gamma) \leq \epsilon } \frac{ \ell(\gamma)}{2 \sinh \frac{k\ell(\gamma)}{2}}\\
&\geq e^{-\frac{\tau+\epsilon}{2}} \sum_{k=1}^{\infty} \sum_{1 < k\ell (\gamma)<L} \frac{\left(1+c\mathrm{Re}\left(\mathrm{Tr}(\chi(\gamma))\right) \right) \ell(\gamma)}{1-e^{-k\ell(\gamma)}} \varphi_{L,\epsilon}^{-}(k\ell(\gamma))-D \frac{1}{\epsilon}\left(g+  \sum_{\gamma \in\mathcal P(X), \ell(\gamma)<\epsilon}\log\left(\frac{\epsilon}{\ell(\gamma)}\right) \right),
\end{align*}
for some absolute constant $D$.
Therefore, we deduce the analogue of the upper bound \eqref{eq. upper bound}, namely that
\begin{align}\label{eq. upper bound 1}\nonumber
\sum_{k=1}^{\infty} &\sum_{1 < k\ell (\gamma)<L} \frac{\left(1+c\mathrm{Re}
\left(\mathrm{Tr}(\chi(\gamma))\right) \right) \ell(\gamma)}{1-e^{-k\ell(\gamma)}} \varphi_{L,\epsilon}^{-}(k\ell(\gamma) \\ &\leq\frac{1}{2}\sum_{0 \le \widetilde \lambda_j \le \frac{1}{4}} e^{\widetilde s_j \tau}  +  \frac{c}{2}\sum_{0 \le \lambda_j \le \frac{1}{4}} e^{s_j \tau}  +  C_c(\eta)m\left(  g\epsilon e^{\tau} + \frac{g}{\epsilon^2}e^{\tau/2} + \frac{e^{\tau/2}}{\epsilon} \sum_{\gamma \in\mathcal P(X), \ell(\gamma)<\epsilon}\log\left(\frac{\epsilon}{\ell(\gamma)}\right) \right),
\end{align}
for some absolute constant $ C_c(\eta)$ depending upon $c$ and $\eta$.

The analogue of the lower bound \eqref{eq. lower bound} is deduced similarly, resulting in the inequality that
\begin{align}\label{eq. lower bound 1}
&\sum_{k=1}^{\infty} \sum_{1 < k\ell (\gamma)<L} \frac{\left(1+c\mathrm{Re}\left(\mathrm{Tr}(\chi(\gamma))\right) \right)   \ell(\gamma)}{1-e^{-k\ell(\gamma)}} \varphi_{L,\epsilon}^{+}(k\ell(\gamma)) \\ &\geq \frac{1}{2}\sum_{0 \le \widetilde\lambda_j \le \frac{1}{4}} e^{\widetilde s_j \tau}+ \frac{c}{2}\sum_{0 \le \lambda_j \le \frac{1}{4}} e^{s_j \tau} -B_c(\eta)m\left(  g\epsilon e^{\tau} + \frac{g(1+\epsilon \tau)}{\epsilon^2}e^{\tau/2} + \frac{e^{\tau/2}}{\epsilon} \sum_{\gamma \in\mathcal P(X), \ell(\gamma)<\epsilon}\log\left(\frac{\epsilon}{\ell(\gamma)}\right) \right),\nonumber
\end{align}
for some absolute constant $ B_c(\eta)$ depending upon $c$ and $\eta$.

Upon integrating the upper and lower bounds from $1-\epsilon$ to $L+\epsilon$ with respect to $\tau$ and proceeding as above by
taking large $L$ and the same $\epsilon$ and $x$, one gets that
$$
\Psi_c(x,\chi)= \sum_{\widetilde\lambda_j\leq 1/4}m(\widetilde \lambda_j)\frac{x^{\widetilde s_j}}{\widetilde s_j}+ c \sum_{\lambda_j\leq 1/4}m( \lambda_j)\frac{x^{ s_j}}{s_j} + \mathcal G_c(x;X,\chi)
$$
where
$$|\mathcal G_c(x; X,\chi)|\leq \mathcal C_1(c)mg x^{5/6} + \mathcal C_2(c) mg x^{2/3} \max\left\{ 0, \log\left(\frac{1}{{\rm sys}(X) x^{1/6}}\right)\right\},
$$
for some absolute constants  $\mathcal C_1(c)$ and $ \mathcal C_2(c)$.

Finally, the proof of Theorem \ref{th. PGT2} concludes by observing that
$$
\Psi(x,\chi)=\frac{1}{c}\left(\Psi_c(x,\chi) - \Psi(x,{\rm Id}_1) \right).
$$
\end{proof}

\section{Asymptotics of spectral determinants}\label{sec:spectral_dets}

In this section we will use Theorem \ref{th. PGT2} and study spectral determinants.

Let $X_{n}:= \overline{\Gamma}_n\backslash \mathbb{H}$ be
a sequence of compact Riemann surfaces of genus $g_n$ with $\lim_{n\to\infty} g_n=\infty$. This implies that
\begin{equation}\label{eq. assumpt on surface main}
\lim_{n\to\infty}{\rm vol}(X_n)=\lim_{n\to \infty}2\pi (2g_n-2)=\infty.
\end{equation}
We consider the associated sequence of weighted Laplacians $\Delta_{2k_n}$,  with the sequence of eigenvalues $k_n(1-k_n)=\lambda_{0,n}< \lambda_{1,n}< ...$ tending to $+\infty$ with multiplicities $ m(\lambda_{j,n})$ and each with an associated $m\times m$ unitary multiplier system $\chi_n$ of weight
$2k_n\in(0,2)$. We denote by $\mathcal{N}_{k,n}$ the number of eigenvalues $\leq 1/4$ of the weighted Laplacian.  In addition, we assume that the following
two conditions are fulfilled.

\medskip
\begin{itemize}
  \item [(i)]({\bf Weak spectral gap}) There exists a constant $\beta < \infty$ such that
\begin{equation}\label{ass limsup 1}
\limsup_{n\to\infty} \frac{\mathcal N_{k,n}}{\lambda_{1,n}{\rm vol}(X_n)}=\beta.
\end{equation}
    \item [(ii)] ({\bf Uniform discreteness}) Let
\begin{equation}\label{ass limsup 2}
    1<\delta_n = \min\limits_{P\in\Gamma_n,\, P\neq \mathrm{Id}}\{N(P)\}.
\end{equation}
    Then there exists $\eta>1$ such that $\delta_n\geq \eta$ for all $n\geq 1$.
\end{itemize}

\medskip
\begin{remark}\rm \label{rm:det_and_random}
If $\lambda_{1,n}$ and $\widetilde\lambda_{1,n}$ are  bounded
from below uniformly in $n$ by some non-zero constant $\widetilde \beta$, then from the bound \eqref{eq. number eigenv trivial bound 1} for
$\mathcal N_{k,n}$ we have that equation \eqref{ass limsup 1} follows.

The weak spectral gap assumption \eqref{ass limsup 1} is related to the Buser-Sarnak example \cite{BS94} which proves that Benjamini-Schramm convergence
of the sequence of compact surfaces $\{X_n\}$ of genus $g_n$ does not necessarily imply the existence of a spectral gap.
Specifically, Buser and Sarnak constructed a Benjamini-Schramm convergent sequence with the first eigenvalues $\lambda_{1,n}$ bounded
by a constant times $\log g_n/g_n$.  This upper bound is not sufficient to deduce that the weak spectral gap fails to hold true for such a sequence.

If one is interested in a probabilistic convergence result
for a sequence of ``typical'' compact Riemann surfaces in the large volume regime, a uniform lower $\widetilde \beta$ exists with probability
one in each of the following three models of random compact Riemann surfaces: In the Weil-Petersson model, by the work of Mirzakhani \cite{Mi13}
and later improved in \cite{WX22, LW24, AM25, AM26}; in the Brooks–Makover model, by \cite[Theorem 2.2. (a)]{BM04}; and in the
random cover model by \cite{MNP22}.  The interested reader is referred to the survey article \cite{MN26} for an overview of results
related to the spectral gap conjecture.

Our aim is to prove a deterministic result, hence we find it worthwhile to significantly weaken the standard spectral gap assumption,
by which we mean the assumption that the sequence of the first eigenvalues is uniformly bounded from below.  Additionally,
we note that that $\lambda_{1,n}$ satisfies the inequalities
$$
\frac{h(X_n)^2}{4}\leq \lambda_{1,n}\leq 2h(X_n)(h(X_n)+5)
$$
where $h(X_n)$ is the Cheeger constant (see \cite{Ch70} and \cite{Bu82}).  Given this, one could
state an alternative to \eqref{ass limsup 1} in the form that
$$
\liminf h(X_n)^2 {\rm vol}(X_n) >0
\,\,\,\,\,
\text{\rm as $g_n\to\infty$.}
$$
However, we find this formation less appealing than \eqref{ass limsup 1}.
\end{remark}

\begin{remark}\rm\label{rm:unif disc}
The uniform discreteness assumption \eqref{ass limsup 2} is equivalent to the assertion that
$$
{\rm sys}(X_n)\geq \log \eta >0
\,\,\,\,\,
\text{\rm for all $n$.}
$$
According to Corollary 4.3 of \cite{Mi13}, in the Weil-Petersson model it is expected that
$1/{\rm sys}(X_n)$ is close to $1$ as $g_n\to\infty$.  In other words, the ``typical'' hyperbolic Riemann surface is expected
to fulfill the uniform discreteness assumption. In the  Brooks–Makover model, the uniform
lower bound for the systole exists with probability one; see \cite[Theorem 2.2. (c)]{BM04}.

Moreover, \eqref{ass limsup 2} is related to spectral density of the eigenvalues of the Laplacian
through the work of  Le Masson and Sahlsten \cite{LMS17} and Monk \cite{Mo22}.  From these works,  it is proved
that by assuming \eqref{ass limsup 2} for a sequence $\{X_n\}$ of compact surfaces converging in
Benjamini-Schramm sense to $\mathbb H$, one can conclude that the spectral density (normalized by volume) of the Laplacian on
$X_n$ converges to the spectral measure of the Laplacian on $\mathbb H$.  As proved in \cite{GK24},
Benjamini–Schramm convergence without the uniform discreteness assumption does not imply spectral convergence.
\end{remark}

Following \cite{naud2023determinants} we will set the following definition.  The sequence $\{X_n\}_{n\geq 1}$ is said to satisfy the {\bf hypothesis
$\mathcal{H}(C,L,\kappa)$} for some $C,L>0$ and $0<\kappa<1$, all of which are independent of $n$, if
$$
N_{X_n}(L):=|\{P\in\Gamma_n: N(P)\leq L\}| \leq C{\rm vol}(X_n)^{\kappa},
$$
where $|A|$ denotes the cardinality of the finite set $A$.  One can view this assumption as a type of non-accumulation of geodesics
of a bounded length.

We state and prove now the main theorem of this section.

\begin{theorem} \label{th: main}Let $\{X_n\}$ be a sequence of Riemann surfaces of finite volume satisfying the
weak spectral gap \eqref{ass limsup 1} and uniform discreteness \eqref{ass limsup 2} assumptions.
Let $\{\chi_n\}$ be an associated sequence of $m\times m$ unitary multiplier systems of  weights $2k_n\in(0,2)$, and assume that
$\lim\limits_{n\to \infty} k_n =\alpha\in[0,1)$.
If $\alpha =0$, we further assume that the multiplicities $m(\lambda_{0,n})$ of the first eigenvalue of the Laplacian $\Delta_{2k_n}$
on $X_{n}$ are bounded uniformly in $n$.  Assume there is a universal constant $C$ and
for all $\varepsilon>0$ there exists an $L_\varepsilon >0$, of order $\log L_{\varepsilon} = O(1/\varepsilon)$ and will be determined
precisely below, such that $\mathcal{H}(C,L_\varepsilon,\kappa)$ holds.  Then, with all this, we have for $n$ sufficiently large
\begin{equation} \label{eq. main statement}
\bigg| \frac{2\pi \log\mathrm{det}(\Delta_{2k_n})}{m{\rm Vol}(X_{n})}  - C_{\alpha}\bigg| <\varepsilon,
\end{equation}
where
\begin{align}\label{eq:alpha_constant}\nonumber
C_{\alpha} &=
2\zeta'(-1) +\log\sqrt{2\pi}-\frac{1}{4}+\left(\frac{1}{2}+\alpha\right)
\log\Gamma\left(1+\alpha\right) \\&+ \left(\frac{1}{2}-\alpha\right) \log \Gamma\left(1-\alpha\right)
 - \log G\left(2+\alpha\right)-\log G\left(2-\alpha\right).
\end{align}
\end{theorem}

\begin{remark}
Before we proceed with the proof, let us note that \eqref{eq. main statement} with $\alpha=0$ and $m=1$ agrees with the deterministic result of Naud \cite[Theorem 1.1.]{naud2023determinants}; see also \cite{HW25} for a more general probabilistic convergence result in the  Weil-Petersson model. Namely, in
both papers we have, in their models of convergence, the limiting value
\begin{equation}\label{eq:old_results}
\lim\limits_{n\rightarrow \infty}\frac{2\pi \log\mathrm{det}(\Delta_{0, X_n})}{{\rm Vol}(X_{n})}=
2\zeta'(-1) +\log\sqrt{2\pi}-\frac{1}{4}.
\end{equation}
Because $\Gamma(1)=G(2)=1$, then \eqref{eq:old_results} is obtained from \eqref{eq. main statement} when $\alpha=0$.
\end{remark}

\begin{proof}
Our starting point is \eqref{e:det_Delta} with $\chi=\chi_n$.  Fix a real number $s>1$ such that $k-s\notin \mathbb{Z}_{\geq 0}$,
which means that the right-hand side of \eqref{e:det_Delta} is non-zero.  Then
$$
\log\mathrm{det}(\Delta_{2k_n}-s(1-s))=\log Z(s;\chi_n)+\log Z_I(s;\chi_n)+ \tilde c.
$$
If $k_n\in(0,1)$, then the smallest eigenvalue of $\Delta_{2k_n}$ is positive; see \cite{Hejhal83}, page 370 with $k_n=m/2$.
As such, we have that $Z(1;\chi_n)\neq 0$. Moreover, $Z_I(1;\chi_n)$, is well defined and nonzero. Therefore
\begin{align*}
\frac{2\pi \log\mathrm{det}(\Delta_{2k_n})}{m{\rm Vol}(X_{n})}&= \lim_{s\to 1^+}\frac{2\pi\log Z(s;\chi_n)}{m{\rm Vol}(X_{n})}  \\ &+\log\sqrt{2\pi} - \frac{1}{4}+2\zeta'(-1)+ \left(\frac{1}{2}+k_n\right) \log\Gamma\left(1+k_n\right) \\
 &+ \left(\frac{1}{2}-k_n\right) \log \Gamma\left(1-k_n\right)- \log G\left(2+k_n\right)-\log G\left(2-k_n\right).
\end{align*}
Given that
\begin{align*}
\lim_{n\to \infty} &\left(\left(\frac{1}{2}+k_n\right) \log\Gamma\left(1+k_n\right) + \left(\frac{1}{2}-k_n\right) \log \Gamma\left(1-k_n\right)- \log G\left(2+k_n\right)-\log G\left(2-k_n\right)\right)
\\& =\left(\frac{1}{2}+\alpha\right) \log\Gamma\left(1+\alpha\right)+ \left(\frac{1}{2}-\alpha\right) \log \Gamma\left(1-\alpha\right)
 - \log G\left(2+\alpha\right)-\log G\left(2-\alpha\right),
\end{align*}
the difference between the two quantities can be made smaller than $\varepsilon/2$, for $n$ large enough. Therefore, in order to prove
\eqref{eq. main statement}, it suffices to show there exists $L_\varepsilon >0$ such that, assuming  $\mathcal{H}(C,L_\varepsilon,\kappa)$ for some $C>0$,
then
\begin{equation}\label{eq. eq. to prove}
\left|\lim_{s\to 1^+}\frac{2\pi\log Z(s;\chi_n)}{m{\rm Vol}(X_{n})} \right| <\frac{\varepsilon}2
\end{equation}
for $n$ sufficiently large.

For any real number $s>1$ such that $k-s\notin\mathbb{Z}_{\geq 0}$, we have from pp. 496--497 of \cite{Hejhal83} that
$$
\log Z(s;\chi_n)= -\sum_{P\neq \mathrm{Id}}\frac{\mathrm{Tr}(\chi_n(P))\Lambda_1(P)}{N(P)^s},
$$
where the sum is taken over \emph{all} hyperbolic classes $P$, not just primitive classes.  As usual, $N(P)$ denotes the norm of $P$ and
$$
\Lambda_1(P)= \frac{\Lambda(P)}{\log N(P)}=\frac{\log N(P_0)}{\log N(P)(1-N(P)^{-1})}
$$
where $P_0$ is the unique primitive hyperbolic element such that $P=P_0^\ell$ for some integer $\ell\geq 1$.
To continue, let us write $\log Z(s;\chi_n)$ as the Stieltjes integral involving $\Psi(x,\chi_n)$.
Specifically, for any sufficiently large and fixed $M$, we can write that
\begin{equation}\label{eq. int up to M}
\sum_{P\neq \mathrm{Id}; N(P)\leq M}\frac{\mathrm{Tr}(\chi_n(P))\Lambda_1(P)}{N(P)^s}
= \int\limits_{\delta_n-\delta}^{M}\frac{d\Psi(t,\chi_n)}{t^s \log t},
\end{equation}
where $\delta>0$ is such that $\eta-\delta>1$ in the notation of \eqref{ass limsup 2}.  To be precise,
\eqref{eq. int up to M} follows from the fact that $\Psi$ is of bounded variation on every finite interval and $t^{-s}\log t$ is continuous.
 Recall from Theorem \ref{th. PGT2} that for any $x>x_{0}$ we have that
\begin{equation}\label{eq. bound prime geod}
  \Psi(x,\chi_n)= \sum_{\lambda_{j,n}\leq 1/4}m(\lambda_{j,n})\frac{x^{s_{j,n}}}{s_{j,n}} + \mathcal G (x;X_n,\chi_n),
\end{equation}
where $\lambda_{j,n}=s_{j,n}(1-s_{j,n})$ and
\begin{equation}\label{eq. error pntbound}
| \mathcal G (x;X_n,\chi_n))|\leq \mathcal C_1mg_n x^{5/6} + \mathcal C_2 mg_n x^{2/3} \max\left\{ 0, -\log\log\eta-\frac{1}{6}\log x\right\}.
\end{equation}
When $k_n\in(0,1)$ we have that $s_{0,n}<1$.  Using this, together with \eqref{eq. bound prime geod} and \eqref{eq. error pntbound}
and after integrating by parts and letting $M\to\infty$,  we get that
$$
\log Z(s;\chi_n)= -\lim_{M\to\infty} \int\limits_{\delta_n-\delta}^{M}\frac{\Psi(t,\chi_n)(s\log t+1)}{t^{s+1}\log^2 t}dt,
$$
so then
\begin{equation}\label{eq. logZ at 1}
\lim_{s\to 1^+}\log Z(s;\chi_n)=-\int\limits_{\delta_n-\delta}^{\infty}\frac{\Psi(t,\chi_n)(\log t+1)}{t^{2}\log^2 t}dt.
\end{equation}

We now study the asymptotic behavior of \eqref{eq. logZ at 1} in $n$.
Consider $L \geq x_{0}$. From \eqref{eq. bound prime geod},  when combined with the bound \eqref{eq. error pntbound}, we have that
\begin{align*}
\left|\int\limits_{L}^{\infty}\frac{\Psi(t,\chi_n)(\log t+1)}{t^{2}\log^2 t}dt\right|&\leq \frac{2}{\log L}\bigg[ \sum_{\lambda_{j,n}\leq 1/4}m(\lambda_{j,n})\frac{L^{s_{j,n}-1}}{\lambda_{j,n}} dt + mg_n \left( 6C_1L^{-\frac{1}{6}} + 3|\log\log \eta|C_2 L^{-\frac{1}{3}} \right)\bigg] .
\end{align*}
Therefore
\begin{align}\label{eq. large L bound}\nonumber
\frac{1}{{\rm vol}(X_n)}\left|\int\limits_{L}^{\infty}\frac{\Psi(t,\chi_n)(\log t+1)}{t^{2}\log^2 t}dt\right|&\leq \frac{2L^{-\lambda_{0,n}}}{\log L}  \bigg[\frac{m(\lambda_{0,n})}{\lambda_{0,n}{\rm vol}(X_n)} + \frac{\mathcal N_{k,n}}{\lambda_{1,n}{\rm vol}(X_n)}\bigg] \\&+ \frac{1}{L^{\frac{1}{6}}\log L} \frac{C(m,\eta)g_n}{{\rm vol}(X_n)},
\end{align}
where we denote by $C(m,\eta)$ an absolute constant which depends only on $m$ and $\eta$.
We can write $\lambda_{0,n}{\rm vol}(X_n)= 4\pi k_n(1-k_n)(g_n-1)$.  Because the weight $k_n$ belongs to the set $\displaystyle \frac{\mathbb Z}{2m(g_n-1)}$,
there is a positive integer $a_{n}$ such that
$k_n=a_n/(2m(g_n-1))$. Hence, we have that
\begin{equation}\label{eq. lambda vol}
\lambda_{0,n}{\rm vol}(X_n)=\frac{2\pi a_n}{m}\left(1-\frac{a_n}{2m(g_n-1)}\right)\geq \frac{2\pi }{m}\left(1-\frac{a_n}{2m(g_n-1)}\right).
\end{equation}
Given that $\lim_{n\to\infty} k_n =\alpha\in[0,1)$, we consider two distinct possibilities. First, if $\alpha>0$, then $\lambda_{0,n}>C_1(\alpha)>0$
for some constant depending only upon $\alpha$ and for all $n\in\mathbb N$. Therefore, using the bound
\eqref{eq. number eigenv trivial bound 1}, we deduce that
$$
\frac{m(\lambda_{0,n})}{\lambda_{0,n}{\rm vol}(X_n)}\leq C(\alpha,\eta),
$$
for some constant depending only upon $\alpha$ and $\eta$.
\noindent
Second, if $\alpha=0$, then \eqref{eq. lambda vol} implies that  $\lambda_{0,n}{\rm vol}(X_n) \geq \frac{\pi}{m}$
for all large enough $n$. If we denote by $M$ the uniform upper bound for the multiplicities $m(\lambda_{0,n})$ we have in this case that
$$
\frac{m(\lambda_{0,n})}{\lambda_{0,n}{\rm vol}(X_n)}\leq \frac{mM}{\pi}.
$$
Therefore, in both cases, the quantity $m(\lambda_{0,n})/(\lambda_{0,n}{\rm vol}(X_n))$ is bounded from above by
some absolute constant $C(m,M,\alpha,\eta)$.  This will show that the upper bound in \eqref{eq. large L bound} goes to zero as $L$ tends to
infinity, as we now discuss.

Fix any $\varepsilon >0$. In view of the weak spectral gap assumption \eqref{ass limsup 1},
we can choose  $L_\varepsilon =L(\varepsilon)$ large enough so that
\begin{equation}\label{eq:L_epsilon_size1}
\frac{2(C(m,M,\alpha,\eta)+\beta)}{\log L_\varepsilon} +  \frac{1}{L_\varepsilon^{\frac{1}{6}}\log L_\varepsilon} \frac{C(m,\eta)g_n}{{\rm vol}(X_n)} <\frac{\varepsilon}{16\pi }.
\end{equation}
The estimate for $L_{\varepsilon}$ given in the statement of the theorem is determined by \eqref{eq:L_epsilon_size1}.
Then \eqref{eq. large L bound} becomes
\begin{equation}\label{eq. bound large L}
\frac{1}{m(2g_n-2)}\left|\int\limits_{L_\varepsilon}^{\infty}\frac{\Psi(t,\chi_n)(\log t+1)}{t^{2}\log^2 t}dt\right|<\frac{\varepsilon}{4}.
\end{equation}

To finish, we now consider the integral portion of \eqref{eq. logZ at 1} for $t < L_{\varepsilon}$.
Assume $\mathcal{H}(C,L_\varepsilon,\kappa)$ holds for some $C>0$. Since $|{\rm Tr}(\chi_n)|\leq m$, we have
\begin{align} \nonumber
\left|\int\limits_{\delta_n-\delta}^{L_\varepsilon}\frac{\Psi(t,\chi_n)(\log t+1)}{t^{2}\log^2 t}dt\right|&\leq m \int\limits_{\delta_n-\delta}^{L_\varepsilon}\frac{ (\log t+1)\sum\limits_{P\in\Gamma_n;\, N(P)\leq t} \Lambda(P)}{t^{2}\log^2 t}dt\\
&\leq m \left(\sum\limits_{P\in\Gamma_n;\,N(P)\leq L_\varepsilon} \Lambda(P)\right) \int\limits_{\eta-\delta}^{L_\varepsilon}\frac{ (\log t+1)}{t^{2}\log^2 t}dt\label{eq. small geod bound}
\\&\leq Bm \sum\limits_{P\in\Gamma_n;\,N(P)\leq L_\varepsilon} \Lambda(P)\nonumber
\end{align}
where
$$
B=\int\limits_{\eta-\delta}^{\infty}\frac{ (\log t+1)}{t^{2}\log^2 t}dt.
$$
Under hypothesis $\mathcal{H}(C,L_\varepsilon,\kappa)$, we get that
$$
\sum_{N(P)\leq L_\varepsilon} \Lambda(P)\leq \sum\limits_{P\in\Gamma_n;\, N(P)\leq L_\varepsilon} \log N(P_0) \leq \log L_\varepsilon N_{X_n}(L_\varepsilon)\leq C \log L_\varepsilon ({\rm vol}(X_n)^{\kappa}).
$$
When combined with \eqref{eq. small geod bound}, the bound becomes
\begin{equation}\label{eq:L_epsilon_size2}
\frac{1}{m(2g_n-2)}\left|\int\limits_{\delta_n-\delta}^{L_\varepsilon}\frac{\Psi(t,\chi_n)(\log t-1)}{t^{2}\log^2 t}dt\right|\leq \frac{2\pi C B \log L_\varepsilon}{{\rm vol}(X_n)^{1-\kappa}}.
\end{equation}
Given that $\kappa\in(0,1)$, there exists $n_1\geq n_0$ depending solely on $\varepsilon$ and such that $n\geq n_1$ implies that
$$
\frac{1}{m(2g_n-2)}\left|\int\limits_{\delta_n-\delta}^{L_\varepsilon}\frac{\Psi(t,\chi_n)(\log t-1)}{t^{2}\log^2 t}dt\right|\leq \frac{2\pi C B \log L_\varepsilon}{{\rm vol}(X_n)^{1-\kappa}}<\frac{\varepsilon}{4}.
$$
Combined with \eqref{eq. small geod bound} and the expression \eqref{eq. logZ at 1} for $\lim_{s\to1+} \log Z(s;\chi_n)$,  this proves that \eqref{eq. eq. to prove} holds true for $n\geq n_1$.

\end{proof}

\begin{remark} The size of $L_{\varepsilon}$ is determined by the inequalities
\eqref{eq:L_epsilon_size1} and \eqref{eq:L_epsilon_size2}, though in \eqref{eq:L_epsilon_size2} we have
that $L_{\varepsilon}$ was determined by \eqref{eq:L_epsilon_size1} and then we used that $\kappa < 1$.  If need
be, further refinement could occur.  For example, one could allow a sequence $\kappa_{n} = \kappa(X_{n})$ which
tend to one provided \eqref{eq:L_epsilon_size2} will be less than $\varepsilon/4$ for $n$ sufficiently large.
\end{remark}

\begin{remark}
From the proof of the above theorem, it is evident that Theorem \ref{th: main} holds
when the weak spectral gap assumption is replaced by the following assumption:

\medskip
\noindent
(iii) ({\bf Uniform boundedness of the number of small eigenvalues}) There exists a constant $\mathcal M$ such that
$$
\mathcal{N}_{k,n}\leq \mathcal M, \quad \text{ for all   } n\in\mathbb N.
$$
\medskip

\noindent
With the notation as above, and assuming the uniform boundedness of the number of small eigenvalues we have for large enough $n$ that
$$
\frac{\mathcal{N}_{k,n}}{\lambda_{1,n}{\rm vol}(X_{n})}\leq \frac{\mathcal{M}}{\lambda_{0,n}{\rm vol}(X_{n})}\leq \frac{m \mathcal M}{\pi (1-\alpha)}.
$$
\end{remark}

A particularly interesting circumstance is when $\alpha=1/2$. Such a limit can be reached in many different ways, such as when $k_n=\frac{1}{2}=\frac{m(g_n-1)}{m(2g_n-2)}$, which is an admissible value for all $n$. In that case we have
the following corollary.

\begin{corollary}
  Let $\{\chi_n\}_{n\geq 1}$ be a sequence of $m\times m$ unitary multiplier systems of weight $1$ associated to a sequence of
  compact surfaces $\{X_n=\overline{\Gamma}_n \backslash \mathbb{H}\} $ such that the minimal norm of the hyperbolic element
  from each $\Gamma_n$ is uniformly bounded from below by some $\eta>1$. Then,
  for any $\varepsilon > 0$ with assumptions on
  $L_{\varepsilon}$ and $\mathcal{H}(C,L_\varepsilon,\kappa)$ as in Theorem \ref{th: main}, for sufficiently large $n$ we have that
  $$
 \left| \frac{2\pi \log\mathrm{det}(\Delta_{1})}{m{\rm Vol}(X_{n})} +\zeta'(-1) -\frac{5}{12}\log 2+\frac{1}{4}\right|<\varepsilon.
$$
\end{corollary}
\begin{proof}
  The proof follows from \eqref{eq. main statement} with $\alpha=1/2$, combined with the functional equation \eqref{eq. funct we. Barnes} for the Barnes double gamma function and the explicit evaluation
  $$
  G\left(\frac{3}{2}\right)=\sqrt{\pi} G\left(\frac{1}{2}\right)=2^{\frac{1}{24}}e^{\frac{1}{8}}\pi^{\frac{1}{4}}A^{-\frac{3}{2}},
  $$
  where $A$ is the Glaisher's constant, which is that
  $$
  A=e^C, \quad \text{where}\quad C=\frac{1}{12}-\zeta'(-1);
  $$
  see \cite[Section 2]{AD14} and \cite{Barnes899}
\end{proof}

Another interesting situation arises when $2k_n=\frac{p}{q}\in(0,2)$. Such a value is admissible on any surface of genus $g_n$ for which
$g_n=\ell_n q+1$, for some positive integer $\ell_n$. We have the following corollary.

\begin{corollary}
Fix a rational number $\frac{p}{q}\in(0,2)$ and let $\{X_n=\overline{\Gamma}_n \backslash \mathbb{H}\} $ be a sequence of compact surfaces of
genus $g_n=\ell_n q+1$, where $\ell_n\rightarrow \infty$ monotonically as $n\to\infty$. Assume that $\{X_n\}$ is such that the minimal norm of the hyperbolic
element from each $\Gamma_n$ is uniformly bounded from below by some $\eta>1$.   Let $\{\chi_n\}_{n\geq 1}$ be a sequence of $m\times m$
unitary multiplier systems of weight $\frac{p}{q}$ associated to $\{X_n\}$.  Then,
for any $\varepsilon > 0$ with assumptions on
$L_{\varepsilon}$ and $\mathcal{H}(C,L_\varepsilon,\kappa)$ as in Theorem \ref{th: main}, for sufficiently large $n$ we have that
$$
\bigg| \frac{2\pi \log\mathrm{det}(\Delta_{p/q})}{m{\rm Vol}(X_{n})}  - A_{p,q} \bigg| <\varepsilon.
$$
where
\begin{align*}
A_{p,q} &=
2\zeta'(-1) +\log\sqrt{2\pi}-\frac{1}{4}+\left(\frac{1}{2}+\frac{p}{2q}\right) \log\Gamma\left(1+\frac{p}{2q}\right)
\\& + \left(\frac{1}{2}-\frac{p}{2q}\right) \log \Gamma\left(1-\frac{p}{2q}\right)
 - \log G\left(2+\frac{p}{2q}\right)-\log G\left(2-\frac{p}{2q}\right).
\end{align*}
\end{corollary}

\thispagestyle{empty}
{\footnotesize
\bibliographystyle{amsalpha}
\bibliography{ref}
}

\vspace{3mm}
\noindent
Jay Jorgenson \\
 Department of Mathematics \\
 The City College of New York \\
 Convent Avenue at 138th Street \\
 New York, NY 10031 U.S.A. \\
 e-mail: jjorgenson@mindspring.com

 \vspace{3mm}
\noindent
Lejla Smajlovi\'{c} \\
 Department of Mathematics and Computer Sciences\\
 University of Sarajevo\\
 Zmaja od Bosne 35, 71 000 Sarajevo\\
 Bosnia and Herzegovina\\
 e-mail: lejlas@pmf.unsa.ba

\vspace{3mm}\noindent
Polyxeni Spilioti\\
Department of Mathematics \\
University of Patras\\
Panepistimioupoli Patron 265 04\\
Greece\\
e-mail: pspilioti@upatras.gr

\end{document}